\documentclass{my_aims}

\usepackage{amsmath}
\usepackage{amssymb}
\usepackage{amsthm}
\usepackage{bm}
\usepackage{color}   

\usepackage{graphics} 
\usepackage{epsfig} 
\usepackage{graphicx}  
\usepackage{epstopdf}
\usepackage[colorlinks=true]{hyperref}

\hypersetup{urlcolor=blue, citecolor=red}

\def\currentvolume{X}
 \def\currentissue{X}
  \def\currentyear{200X}
   \def\currentmonth{XX}
    \def\ppages{X--XX}
     \def\DOI{10.3934/xx.xx.xx.xx}

\newtheorem{theorem}{Theorem}[section]
\newtheorem{proposition}{Proposition}
\theoremstyle{definition}
\newtheorem{definition}[theorem]{Definition}

\usepackage[british,UKenglish,USenglish,english,american]{babel}
\usepackage{ragged2e,array,booktabs}
\usepackage{enumerate}
\usepackage[toc,page]{appendix}

\title[Enhancement of chemotherapy using oncolytic virotherapy]{Enhancement 
of chemotherapy using oncolytic virotherapy: Mathematical and optimal control analysis}

\author[J. Malinzi, R. Ouifki, A. Eladdadi, D.F.M. Torres and K.A.J. White]{}



\email{Josephmalinzi@aims.ac.za}
\email{Rachid.Ouifki@up.ac.za}
\email{eladdada@strose.edu} 
\email{delfim@ua.pt}
\email{K.A.J.White@bath.ac.uk}

\thanks{$^*$ Corresponding authors: Joseph Malinzi \& Amina Eladdadi }

\begin{document}

\maketitle


\centerline{\scshape Joseph Malinzi$^{*1}$, Rachid Ouifki$^1$}
\medskip
{\footnotesize
 \centerline{$^1$ Department of Mathematics and Applied Mathematics, University of Pretoria,}
 \centerline{Private Bag X 20, Hatfield, Pretoria 0028, South Africa}} 

\medskip

\centerline{\scshape Amina Eladdadi$^{*2}$, Delfim F.M. Torres$^3$ and K.A. Jane White$^4$}
\medskip
{\footnotesize
 \centerline{$^2$Department of Mathematics, The College of Saint Rose, Albany, New York, USA}
 \centerline{$^3$Center for Research and Development in Mathematics and Applications (CIDMA),}
 \centerline{ Department of Mathematics, University of Aveiro, 3810-193 Aveiro, Portugal}
 \centerline{$^4$Department of Mathematical Sciences, University of Bath,}
 \centerline{Claverton Down, Bath BA2 7AY, UK}}

\bigskip


\begin{abstract}
Oncolytic virotherapy has been emerging as a promising novel cancer treatment 
which may be further combined with the existing therapeutic modalities 
to enhance their effects. To investigate how virotherapy could enhance chemotherapy, 
we propose an ODE based mathematical model describing the interactions between tumour cells, 
the immune response, and a treatment combination with chemotherapy and oncolytic viruses.
Stability analysis of the model with constant chemotherapy treatment rates shows that 
without any form of treatment, a tumour would grow to its maximum size. It also demonstrates 
that chemotherapy alone is capable of clearing tumour cells provided that the drug efficacy 
is greater than the intrinsic tumour growth rate. Furthermore, virotherapy alone may not be 
able to clear tumour cells from body tissue but would rather enhance chemotherapy if viruses 
with high viral potency are used. To assess the combined effect of virotherapy and chemotherapy 
we use the forward sensitivity index to perform a sensitivity analysis, with respect to chemotherapy 
key parameters, of the virus basic reproductive number and the tumour endemic equilibrium. 
The results from this sensitivity analysis indicate the existence of a \textit{critical} dose 
of chemotherapy above which no further significant reduction in the tumour population 
can be observed. Numerical simulations show that a successful combinational therapy 
of the chemotherapeutic drugs and viruses depends mostly on the virus burst size, 
infection rate, and the amount of drugs supplied. 
Optimal control analysis was performed, by means of the Pontryagin's maximum principle, 
to further refine predictions of the model with constant treatment rates by accounting 
for the treatment costs and sides effects. Results from this analysis suggest that the 
optimal drug and virus combination correspond to half their maximum tolerated doses. 
This is in agreement with the results from stability and sensitivity analyses. 
\end{abstract}


\section{Introduction}

Combination therapy strategies have shown significant promise for 
treating cancers that are resistant to conventional modalities. 
Oncolytic virotherapy  (OV) is an emerging treatment modality that uses replication 
competent viruses to destroy cancers \cite{ES07,SK12,TE07,JPH14}. Their tumour specific 
properties allow for viral binding, entry, and replication \cite{AL14}. 
Various studies have investigated combination strategies with OV and chemotherapeutic drugs 
to optimize both viral oncolysis and the effect of the added therapy 
\cite{binz2015chemovirotherapy,AL14,relph2016cancer}. Oncolytic viruses 
can greatly enhance the cytotoxic mechanisms of chemotherapeutics \cite{PJ10}. 
Furthermore, chemotherapeutic drugs lyse fast multiplying cells and, in general, 
virus infected tumour cells quickly replicate \cite{B12}. In most combination studies 
with chemotherapy drugs, apoptosis was increased but viral replication was not enhanced 
\cite{AL14,UF10,AG07}. For example, Nguyen et al. \cite{AL14} gave an account 
of the mechanisms through which drugs can successfully be used in a combination 
with oncolytic viruses. They, however, noted that the success of this combination 
depended on several factors including the type of oncolytic virus (OV)-drug combination 
used, the timing, frequency, dosage, and cancer type targeted. A review of recent 
clinical studies by Binz and Urlrich \cite{binz2015chemovirotherapy} shows the 
classification of possible combination of drugs and oncolytic viruses. 
Oncolytic virotherapy is still ongoing clinical trials \cite{UF10, AG07}, 
and some virus-drug combos, for example talimogene laherparepvec, 
have been approved for the treatment of melanoma \cite{FDA2017}. 
   
While, there is a growing body of scientific research showing that combination 
therapies are the cutting edge for cancer treatment, the design of an optimal protocol 
remains an open question. The chemo-viro therapy combination is no exception. 
Despite a noticeable success in the clinical and experimental studies to investigate 
and characterise the treatment of cancer with chemotherapeutics-virus combinations, 
much is still not understood about chemovirotherapy. Some of the main open questions 
in designing an optimal chemovirotherapy treatment are figuring out the right dose combination, 
the most effective method of drug infusion, and the important treatment 
characteristics \cite{konstorum2017addressing}. 

Mathematical modelling and optimal control theory have over the years played 
an important role in answering such important research questions which would 
cost much to set up experimentally. Several mathematical studies including 
\cite{LKW09, LH05, SHF02} have addressed the dynamics of oncolytic virotherapy treatment. 
For example, Ursher \cite{JR94} gave a summary of some mathematical models for chemotherapy. 
Tian \cite{JP11} presented a mathematical model that incorporates burst size for oncolytic 
virotherapy. His analysis showed that there are two threshold burst size values and below 
one of them the tumour always grows to its maximum size. His study affirmed that a tumour 
can be greatly reduced to low undetectable cell counts when the burst size is large enough. 
A recent study by Malinzi et al. \cite{malinzi2017modelling} showed that that chemotherapy 
alone was not able to deplete tumour cells from body tissue but rather in unison with 
oncolytic viruses could possibly reduce the tumour concentration to a very low undetectable state. 
Other similar mathematical models on virotherapy include \cite{JPH14,YJ13,D01,AFE06,MA11}. 

In this paper, we address the question on: \textit{``what is the optimal chemotherapeutic 
drug and virus dosage combination for the elimination of tumour cells in body tissue?''} 
To this end, we develop a mathematical model and optimal control problem that account 
for the combination of cancer treatment using chemotherapy and virotherapy.  
The paper is organised as follows. Section~\ref{Model} is devoted to the model 
description and the underlying assumptions. In Section~\ref{modelanalysis}, 
we analyse three sub-models: without-treatment model, chemotherapy-only model, 
and the oncolytic virotherapy-only model. The analysis of the whole model, that is, 
the chemotherapeutic drug combined with the virus therapy is detailed 
in Section~\ref{chemoviro}. In Section~\ref{sec:Optimal_control_analysis}, 
we investigate an optimal control problem with a quadratic objective function 
to determine the optimal dosage combination of the chemotherapy drug and the virotherapy. 
The numerical simulations illustrating our theoretical results are presented 
in Section~\ref{sec:Numerical_simulations}. We conclude with a discussion 
in Section~\ref{conclusion}. 
	

\section{Model Formulation}
\label{Model}

In this section, we propose and formulate our new mathematical model describing the growth 
of an avascular solid tumour under the effect of both chemotherapy and oncolytic  virotherapy 
treatments. The model considers six state variables summarized in Table~\ref{tab:modelvars}. 
We first start by stating the underlying model assumptions in Subsection~\ref{assumptions}, 
followed by the model equations and their detailed description in Subsection~\ref{ModelEquations}.


\subsection{Model description and assumptions}
\label{assumptions}

To model the effects of the chemo-virotherapy on the tumour, we consider the dynamics 
of the following interacting cell populations: tumour cells (both uninfected and virus-infected cells), 
immune cells (both virus and tumour specific immune responses), the free virus 
and the drug concentrations.  Based on the discussion above and the scientific literature 
on the interactions between the uninfected and the virus-infected tumour cells in the presence 
of the oncolytic virotherapy and chemotherapy treatments, the following 
assumptions are made in setting up the model: 
\begin{enumerate}

\item Without treatment, the tumour grows logistically 
with a carrying capacity $K$ \cite{JP11}.

\item Virus infection, chemotherapeutic drug response to the tumour, 
chemokine production and immune cells proliferation are considered to be 
of Michaelis--Menten form to account for saturation 
\cite{de2006modeling,agrawal2014optimal,pinho2013mathematical}.

\item The model accounts for both virus and tumour specific 
immune responses \cite{howells2017oncolytic}.

\item Virus production is a function of virus burst size and the death 
of infected immune cells. The number of viruses therefore increases 
as infected tumour cells density multiplies \cite{JP11}.

\item We consider both virus and tumour specific immune responses. 
Virus-specific immune response is proportional to the infected tumour cell 
numbers whilst the tumour specific immune response is taken to be of 
Michaelis--Menten form to account for saturation in the immune proliferation 
\cite{khajanchi2014stability}. 

\item We consider a case where drug infusion per day is constant. The constant 
infusion rate may relate to a situation where a patient is put on an intravenous 
injection or a protracted venous infusion and the drug is constantly pumped into the body. 
This form of drug dissemination is used on cancer patients who stay in the hospital 
for a long period of time. Higher doses of certain anti-cancer drugs may however lead 
to hepatic veno-occlusive disease, a condition where the liver is obstructed as a result 
of using high-dose chemotherapy \cite{RB83, MHY88}. Another more realistic consideration 
would be to use periodic infusion \cite{MH08,SGM05}. This would lead to a periodic system 
of equations for which standard theory on periodic systems applies 
(see for example \cite{ouifki2007model}). This case is, however, 
not dealt with in this work, and will be studied in a future work.
\end{enumerate}


\subsection{Model equations} 
\label{ModelEquations}

The proposed mathematical model consists of the following system of 
six differential equations (\ref{eq:uninfected})--(\ref{eq:drug}): 
\begin{align}
\frac{dU}{dt} & = \alpha U\left(1-\frac{U+I}{K}\right)
- \frac{\beta UV}{K_u +U} -\nu_U U E_T  - \frac{\delta_U  UC}{K_c +C} \label{eq:uninfected}\\
\frac{dI}{dt} & = \frac{\beta UV}{K_u +U}- \delta I
- \nu_I E_T I -\tau E_V I -\frac{\delta_I IC}{K_c +C} \label{eq:infected} \\
\frac{dV}{dt} & = b\delta I- \frac{\beta UV}{K_u +U}-\gamma V  \label{eq:virus} \\
\frac{dE_V}{dt} & = \phi I - \delta_V E_V \label{eq:immune_v} \\
\frac{d E_T}{dt} & = \frac{\beta_T (U+I)}{\kappa+ (U+I)}-\delta_T E_T \label{eq:immune_T} \\
\frac{dC}{dt} &= g(t) - \psi C \label{eq:drug}
\end{align}
The model variables and parameters, their meaning and base values are summarised 
in Table~\ref{tab:modelvars} and Table~\ref{tab:parametervalues1} respectively. 
\begin{table}[h!]
\centering
\caption{The model variables}
\label{tab:modelvars}
\begin{tabular}{@{}lll@{}} \hline
Variable & Description & Units \\ \hline
$U(t)$   & Uninfected tumour density  & cells per mm$^3$  \\
$I(t)$   & Virus infected tumour cell density & cells per mm$^3$  \\
$V(t)$   & Free virus particles               & virions per mm$^3$ \\
$E_v(t)$ & Virus specific immune response     & cells per mm$^3$ \\
$E_T(t)$ & Tumour specific immune response    & cells per mm$^3$ \\
$C(t)$   & Drug concentration                 & grams per millilitre (g/ml) \\ \hline
\end{tabular}
\end{table}

The initial conditions for the model are:
\begin{equation}
\label{ICs}  
U(0) = U_{0},  I(0) = I_{0} ,  V(0) = V_{0},  E_V(0) = E_{V_0},   
E_T(0) = E_{T_0},  C(0) = C_{0}, 
\end{equation}
where the  constants $U_{0}$, $I_{0}$,  $V_{0}$,  $E_{V_0}$,  $E_{T_0}$,  
and $C_{0}$ denote the initial concentrations of the uninfected tumour, 
virus infected tumour, free virus particles, virus specific immune, 
tumour specific immune, and the chemotherapeutic drug respectively. 
They are assumed to be nonnegative to make sense biologically. 
A detailed description of the model follows. 

\noindent {\bf In Eqs. (\ref{eq:uninfected}) \& (\ref{eq:infected})}, the terms 
${\alpha U\left(1-\frac{U+I}{K}\right)}$ represents tumour growth. The term 
$\beta U V/(K_u+U)$ describes infection of tumour cells by the virus where $\beta$ 
is the infection rate. Drug effect to the uninfected and infected tumour cells is respectively 
described by the terms $\delta_U  U C/(K_c+C)$ and $\delta_I I C/(K_c+C)$ where $\delta_U $ 
and $\delta_I$ are lysis induced rates. $K_u$ and $K_c$ are Michaelis--Menten constants 
and relate to killing rates when the virus and the drug are half-maximal 
\cite{de2006modeling,agrawal2014optimal,pinho2013mathematical}. 
The term $\nu_U U E_T$ represents tumour specific immune response where $\nu_U$ 
is the uninfected lysis rate by tumour specific immune cells. Infected tumour cells 
have a lifespan $1/\delta$ and are killed by both tumour and virus specific immune cells. 
These are described by the terms  $\nu_I E_T I$ and $\tau E_V I$ where $\delta$ 
is the natural death rate, $\nu_I$ and $\tau$ are respectively lysis rates of tumour 
cells by the tumour specific and virus specific immune cells.  

\noindent {\bf In Eq. (\ref{eq:virus})}, the term $b \delta I$ represents virus 
proliferation where $b$ is the virus burst size and $\delta$ is the infected tumour 
cell death rate. The term $\beta UV/(K_u+U)$ represents loss of free virus due 
to infection of the uninfected tumour cells. Virus deactivation in body tissue 
is represented by the term $\gamma V$ where $\gamma$ is the rate of decay. 

\noindent {\bf In Eqs. (\ref{eq:immune_v}) \& (\ref{eq:immune_T})}, $\phi I$ is the virus 
specific immune production and $\delta_v E_v$ represents its deactivation where $\delta_v$ 
is the decay rate. The term $\beta_T (U+I)/(k+(U+I))$ represents tumour-specific immune 
response where $\beta_T $ is the rate of production and $k$ is a Michaelis constant. 
Tumour-specific immune decay is described by the term $\delta_T E_T$ 
where $\delta_T$ is the rate of decay. 

\noindent {\bf In Eq. (\ref{eq:drug})}, the time dependent function $g(t)$ represents 
drug infusion into the tumour and $\psi C$ describes drug decay 
where $\psi$ is the rate of decay. 


\begin{table}
\centering
\caption{The model parameters, their description and base values.}
\label{tab:parametervalues1}
\scalebox{0.93}{\begin{tabular}{@{}llll@{}} \hline
Symbol & Description &  Value \& units  & Ref.  \\ \hline    
$K$  & Tumour carrying capacity & $10^6$ cells per mm$^3$ per day & \cite{ZTK08} \\
$\alpha$  & Tumour growth rate & $0.206$ cells per mm$^3$ per day & \cite{ZTK08}\\
$\beta$  & Infection rate of tumour cells  & $0.001-0.1$ cells per mm$^3$ per day&   \cite{ZTK08}\\
$\delta$ & Infected tumour cells death & $0.5115$ day$^{-1}$ &  \cite{ZTK08}\\
$\gamma$ & Rate of virus decay & $0.01$ day$^{-1}$  & \cite{ZTK08} \\
$b$ & Virus burst size & $0-1000$ virions per mm$^3$ per day & \cite{BD90} \\
$\psi$ & Rate drug decay & $4.17$ millilitres per mm$^3$ per day & \cite{SDP}  \\
$\delta_U$ & Lysis rate of $U$ by the drug & $50$ cells per mm$^3$ per day & \cite{SDP} \\
$\delta_I$ & Lysis rate of $I$ by the drug & $60$ cells per mm$^3$ per day& \cite{SDP} \\
$\phi$ & $E_V$ production rate & $0.7$ day$^{-1}$ & \cite{bollard2016t}\\
$ \beta_T$ & $E_T$ production rate & $0.5$ cells per mm$^3$ per day& \cite{gajewski2013innate,le2014mathematical} \\
$ \delta_v,\delta_T$ & immune decay rates & $0.01$ day$^{-1}$ & \cite{gajewski2013innate,le2014mathematical} \\
$ K_u,K_c,\kappa$ & Michaelis--Menten constants & $10^5$ cells per mm$^3$ per day & \cite{kirschner1998modeling} \\
$ \nu_U$ & Lysis rate of $U$ by $E_T$& $0.08$ cells per mm$^3$ per day & est\\
$ \nu_I$ & Lysis rate of $I$ by $E_T$& $0.1$  cells per mm$^3$ per day & est\\
$ \tau$ & Lysis rate of $I$ by $E_V$ & $0.2$ cells per mm$^3$ per day & est\\ \hline
\end{tabular}}
\end{table}


\section{Mathematical analysis}
\label{modelanalysis}  

To better understand the dynamics of the proposed model, we begin by examining the model's 
behaviour about the steady states. This analysis is crucial for identifying the conditions 
necessary to achieve a tumour-free state.  We first show that the model is well posed 
in a biologically feasible domain, and then proceed with a stability analysis of the model 
with constant drug infusion. We then analyse the model in the case of no treatment, 
with chemotherapy alone, and with oncolytic virotherapy alone. The case of the chemotherapeutic 
drug combined with the virus therapy is  detailed  as well. 


\subsection{Well-posedness}
\label{Well-posed} 

Before we proceed with the mathematical analysis, we need to show that the model 
is well posed in a biologically feasible domain. The model describes the temporal 
evolution of cell populations, therefore, the cell densities should remain non 
negative and bounded. The well-posedness theorem is stated below and the proofs 
are given in Appendix~\ref{appendixA}. 

\begin{theorem}
\label{thm:welposedness}
\begin{enumerate}
\item[(i)] There exists a unique solution to the system of equations 
(\ref{eq:uninfected})--(\ref{eq:drug}) in the region  
$\mathbb{D} =(U,I,V,E_V,E_T,C) \in \mathbb{R}^{6}_+$.

\item[(ii)] If $U(0)\geq 0$, $I(0)\geq 0$, $V(0)\geq 0$, $E_V(0)\geq 0, E_T(0)\geq 0$ 
and $C(0)\geq 0$, then $U(t)\geq 0$, $I(t)\geq 0$, $V(t)\geq 0$, $E_V(t)\geq 0$, 
$E_T(t) \geq 0$ and $C(t)\geq 0$ for all $t\geq 0$.

\item[(iii)] The trajectories evolve in an attracting region 
$\Omega=\{(U,I,V,E_V,E_T,C)\in \mathbb{D}\mid U(t)+I(t)\leq K,~ V(t)
\leq b/\gamma,~ E_V(t)\leq \phi/\delta_V,E_T\leq \beta_T/\delta_T,C(t)\leq C^*\}$, 
where $C^*$ is the maximum drug concentration and depends on the infusion method.

\item[(iv)] The domain $\Omega$ is positive invariant for the model Equations 
(\ref{eq:uninfected})--(\ref{eq:drug}) and therefore biologically 
meaningful for the cell concentrations.
\end{enumerate}
\end{theorem}

We will analyse the model quantitative behaviour in the domain $\Omega$. 


\subsection{Without-treatment model}
\label{sec:notreatment}

To investigate the efficacy of each treatment, their combination as well as the immune response 
to tumour cells, we first study the model without any form of treatments. The model system of 
eqs. (\ref{eq:uninfected})--( \ref{eq:drug}) without treatment reduces to the following system 
describing the interactions of the uninfected tumour with the tumour specific immune cells:
\begin{align}
\label{eq:no_treatment}
\frac{dU}{dt} & = \alpha U\left(1-\frac{U}{K}\right) - \nu_UUE_T, \nonumber \\
\frac{d E_T}{dt} & = \frac{\beta_T U}{\kappa+ U}-\delta_T E_T.
\end{align}

\begin{proposition}
\label{prop:1}
The model (\ref{eq:no_treatment}) has two biologically meaningful steady states: 
a tumour-free state, $X_1=(U^*=0,E_T^* = 0)$, which is unstable and a tumour endemic state, 
{\small 
\begin{align}
\label{eq:steadystates1:a}
X_2 = \bigg[ U^* 
& =   \frac{K}{2\alpha}\bigg(\alpha - \frac{\alpha \kappa}{K}-M \bigg) 
+ \frac{K}{2\alpha}\sqrt{\bigg(\alpha - \frac{\alpha \kappa}{K}-M \bigg)^2 
+4\frac{\alpha^2\kappa}{K}},\quad E_T^*  
= \frac{\beta_T}{\delta_T} \bigg( \frac{U^*}{\kappa + U^*}\bigg)\bigg],
\end{align}}
where $M = \nu_U \beta_T/\delta_T$, which is locally asymptotically stable.
\end{proposition}

\begin{proof} 
Equating Eqs. (\ref{eq:no_treatment}) to zero yields:
\begin{align}
& \frac{\alpha }{K}U^{*^2}+ \bigg(\frac{\alpha\kappa}{K} + M - \alpha \bigg)U^*  
-\alpha\kappa = 0 \label{eq:showequilibria11}\\
E_T^*  & = \frac{\beta_T}{\delta_T} \bigg( \frac{U^*}{\kappa + U^*}\bigg)  
\label{eq:showequilibria12}.
\end{align}
From which if $U^* = 0$ then $E_T^* = 0$. The eigenvalues of the Jacobian matrix of the model 
(\ref{eq:no_treatment}) evaluated at $(0,0)$ are $\lambda_1= \alpha$ and $\lambda_2=-\delta_T$. 
One positive and the other negative. Implying that the tumour free equilibrium is unstable. 
Compare Equation (\ref{eq:showequilibria11}) to $a U^{*^2} + bU^* + c =0$. We observe that 
$c: = -\alpha \kappa <0$ implying that $\Delta = b^2-4ac: = \bigg(\alpha 
- \left(  \frac{\alpha \kappa}{K}+M \right)\bigg)^2 +4\frac{\alpha^2\kappa}{K}>0$ 
and therefore there is only one positive root of the quadratic equation 
(\ref{eq:showequilibria11}) and its given by 
\[
U^* = \frac{-b+\sqrt{b^2-4ac}}{2a}:  = \frac{K}{2\alpha}\bigg(\alpha 
- \frac{\alpha \kappa}{K}-M \bigg) + \frac{K}{2\alpha}\sqrt{\bigg(\alpha 
- \frac{\alpha \kappa}{K}-M \bigg)^2 +4\frac{\alpha^2\kappa}{K}}.
\]
The characteristic polynomial of the Jacobian matrix evaluated 
at $X_2$ is given by 
\begin{align}
\label{eq:char1:a}
P(x)  = x^2 + P_1x + P_0.
\end{align}
We prove in Appendix~\ref{appendixB}, using Routh--Hurwitz criterion, 
that $P_ 1>0$ and $P_0>0$ and thus the endemic steady state, 
$X_2$, is locally asymptotically stable. 
\end{proof}  

\emph{Biological interpretation:}
Proposition~\ref{prop:1} suggests that, without any form of treatment, the tumour-free state 
is always unstable and the endemic state is stable implying that 
the tumour would not be eliminated from the body. 

Next, the model (\ref{eq:uninfected})--(\ref{eq:drug}) is studied with only 
chemotherapy to investigate the effect of the chemotherapeutic drug on tumour cells. 


\subsection{Chemotherapy-only model}
\label{sec:chemo-only_model}

We consider a case where drug infusion per day is constant, that is, 
$g(t) = q$. The model with chemo-only, that is $I=V=E_V=0$, reduces to:
\begin{align}
\label{eq:chemo_model}
\frac{dU}{dt} & = \alpha U\left(1-\frac{U}{K}\right) 
-\nu_U U E_T  - \frac{\delta_U  UC}{K_c +C}, \nonumber \\
\frac{d E_T}{dt} & = \frac{\beta_T U}{\kappa+ U}
-\delta_T E_T, \nonumber \\
\frac{dC}{dt} & = q - \psi C.
\end{align}

\begin{proposition}
\label{prop:2}
The chemo-only model (\ref{eq:chemo_model}) has two biologically 
meaningful steady states: a tumour-free sate, $C_1=(U^* = 0,E_T^* = 0,C^* 
= \frac{q}{\psi})$, which is stable if $ \alpha K_u 
> \frac{q(\delta_U-\alpha)}{\psi}$ and a tumour endemic state,
\begin{align}
\label{eq:steadystates1}
C_2 = \bigg[ U^* & =  \frac{K}{2\alpha}\bigg(\alpha 
-  \left(  \frac{\alpha \kappa}{K}+L+M \right)\bigg) 
+ \frac{K}{2\alpha}\sqrt{\bigg(\alpha -  \left(  \frac{\alpha \kappa}{K}
+L+M \right)\bigg)^2 +4\frac{\alpha\kappa}{K}(\alpha-L)} \nonumber\\
E_T^* & = \frac{\beta_T}{\delta_T} \bigg( \frac{U^*}{\kappa + U^*}\bigg)
\quad \text{and}\quad C^* = \frac{q}{\psi}\bigg],
\end{align}
where 
\[
L = \frac{\delta_U C^*}{K_c + C^*},\quad M = \frac{\nu_U \beta_T}{\delta_T}.
\]
The endemic state, $C_2$, is stable if the conditions 
\eqref{eq:conditions-forstabilityofc2} in Appendix~\ref{appendixB} are satisfied. 
\end{proposition}

\begin{proof}
Setting Equations in (\ref{eq:chemo_model}) to zero gives:
\begin{align}
\label{eq:showequilibria}
& \alpha U^*\left(1-\frac{U^*}{K} \right) - \nu_U U^*E_T^* 
- \frac{\delta_U U^* C^*}{K_c + C^*}= 0, \\
&  \frac{\beta_T U}{\kappa+U^*} -\delta_T E_T^* U^* =0,
\quad C^* = \frac{q}{\psi},
\end{align}
from which if $U^* = 0$ then $ C_1 = (U^* = 0, E_T = 0, C^* = \frac{q}{\psi} )$. 
The eigenvalues of the Jacobian of (\ref{eq:chemo_model}) evaluated at $C_1$ are 
$\lambda_1 = -\psi$, $\lambda_2 = -\delta_T$ and 
$\lambda_3 = \alpha- \frac{\delta_U q/\psi}{K_u +q/\psi}$ implying that for 
$\lambda_3$ to be negative, then $ \alpha K_u > \frac{q(\delta_U - \alpha)}{\psi}$. 
If $U^* \neq 0$ then from Equation (\ref{eq:showequilibria}), 
\[
E_T^* = \frac{\beta_T}{\delta_T} \bigg( \frac{U^*}{\kappa + U^*}\bigg)\quad \text{and} 
\quad \frac{\alpha }{K}U^{*^2}+ \bigg(\frac{\alpha\kappa}{K} + M 
+L - \alpha \bigg)U^*  + \kappa (L-\alpha ) = 0 
\] 
which when solved, and using the same analysis as in Proposition~\ref{prop:1}, 
gives the expressions in Equation (\ref{eq:steadystates1}).
The characteristic polynomial of the Jacobian matrix evaluated at 
$C_2$ is given in Equation \eqref{eq:chapolyc2}. It directly follows 
from Routh--Hurwitz criterion that the tumour endemic state, $C_2$, 
is locally asymptotically stable if the conditions 
\eqref{eq:conditions-forstabilityofc2} are fulfilled. 
\end{proof}

\emph{Biological interpretation:}
Proposition~\ref{prop:2} suggests that a tumour can be eradicated 
by chemotherapy from the body tissue if the tumour growth rate 
is less than the drug efficacy ($\alpha < \delta_U$). Otherwise, 
it would continue to grow uncontrollably. \\ 
\noindent
Since drug dosage, $q$, must not exceed the Maximum Tolerable Dose (MTD), 
denoted by $q^{MTD}$, we further derive the following result in Proposition~\ref{prop:3}:
 
\begin{proposition}
\label{prop:3} 
\begin{enumerate}
\item If $\delta_U < \alpha$, the tumour-free state 
will always be unstable for any $q < q^{MTD}$.

\item If $\delta_U > \alpha$:
\begin{enumerate}
\item[(i)] If $q^{MTD} < \frac{\alpha \psi K_u }{\delta_U - \alpha}: = q^*$, 
then the tumour-free state is unstable for all $q<q^{MTD}$. 

\item[(ii)] If $q^{MTD} \geq q^*$, then the tumour-free state 
is locally asymptotically stable for all $q\in [q^*,q^{MTD}]$.
\end{enumerate}
\end{enumerate}
\end{proposition} 

\begin{proof}
\begin{enumerate}
\item From the condition for stability in Proposition~\ref{prop:2}, 
if $\delta_U < \alpha$ then $\alpha K_u >  \frac{q\psi}{\delta_U - \alpha}$, 
which violates the condition for stability. 

\item If $\delta_U > \alpha$:
\begin{enumerate}
\item[(i)] If $q^{MTD} <  \frac{\alpha \psi K_u }{\delta_U - \alpha}: = q^*$, 
yet the drug dosage, $q$, can not be larger  than the MTD, $q^{MTD}$, implying 
that the condition for stability in Proposition \ref{prop:2} is still violated. 
   
\item[(ii)] If $q^{MTD} >  q^*$, and $q\in [q^*, q^{MTD}]$ 
then the stability condition if fulfilled. 
\end{enumerate}
\end{enumerate}
\end{proof} 

\emph{Biological interpretation:} Proposition~\ref{prop:3} suggests the followings:
\begin{enumerate}
\item For a chemotherapeutic drug that is not highly efficient 
($\delta _U < \alpha$) the tumour cannot be eradicated. 

\item If the efficacy, $\delta_U$, which is also the lysis rate measured 
in number of uninfected tumour cells lysed per mm$^3$ per day, is high enough but with:
\begin{enumerate}
\item[(i)] a very toxic drug, then one cannot give enough dose to eradicate the tumour 

\item[(ii)] a drug which is less toxic, then one can afford to give a dose which 
is not larger than the MTD and yet which allows for the condition for stability.
\end{enumerate}
\end{enumerate}


\subsection{Virotherapy-only model}
\label{sec:virotherapy-only_model}

To determine the effect of virotherapy on tumour cells, we now study the model 
with only virotherapy treatment. The model (\ref{eq:uninfected})--(\ref{eq:drug}) 
with only virotherapy treatment, that is, $C=0$, reduces to:
\begin{align}
\label{eq:viro-onlymodel}
\frac{dU}{dt} & = \alpha U\left(1-\frac{U+I}{K}\right)
- \frac{\beta UV}{K_u +U} -\nu_U U E_T, \nonumber \\
\frac{dI}{dt}& = \frac{\beta UV}{K_u +U}- \delta I
- \nu_I E_T I -\tau E_V I, \nonumber \\
\frac{dV}{dt} & = b\delta I- \frac{\beta UV}{K_u +U}
-\gamma V, \nonumber \\ 
\frac{dE_V}{dt} & = \phi I - \delta_V E_V, \nonumber \\
\frac{d E_T}{dt} & = \frac{\beta_T (U+I)}{\kappa+ (U+I)}-\delta_T E_T.
\end{align}

\begin{proposition}
\label{prop:4}
The virotherapy-only model \eqref{eq:viro-onlymodel} has three biologically 
meaningful steady states: a tumour-free state $V_1 = \left(U^*=0,I^* = 0, 
V^* = 0, E_T^* = 0, E_V^* = 0 \right)$ which is unstable, a virus-free state  
\begin{align}
\label{eq:steadystates3}
V_2= \bigg[ U^* & =\frac{K}{2\alpha}\bigg(\alpha -  \frac{\alpha \kappa}{K}-M \bigg) 
+ \frac{K}{2\alpha}\sqrt{\bigg(\alpha -    \frac{\alpha \kappa}{K}-M \bigg)^2 
+4\frac{\alpha^2\kappa}{K}},\quad I^* =0,  \nonumber  \\
\quad \quad \quad & ~V^* =0,\quad E_T^* 
= \frac{\beta_T}{\delta_T} \bigg( \frac{U^*}{\kappa + U^*}\bigg), \quad ~E_V^* = 0 \bigg]
\end{align}
and a tumour endemic state $V_3 = (U^*,I^*,V^*,E_T^*,E_V^*)$.  
\end{proposition} 

\begin{proof}
The above expressions are obtained when model equations in \eqref{eq:viro-onlymodel} 
are equated to zero. The eigenvalues of the Jacobian matrix evaluated at $V_1$ are 
$\lambda_1 = \alpha$, $\lambda_2 = -\delta$, $\lambda_3 = -\gamma$, 
$\lambda_4 = -\delta_T$ and $\lambda_5 = -\delta_V$ with one positive and the 
others negative implying that the tumour-free state, $V_1$, is unstable. 
\end{proof}

\emph{Biological interpretation:} Proposition~\ref{prop:4} suggests that 
virotherapy on its own is not capable of eliminating tumour cells. 


\section{Chemovirotherapy}
\label{chemoviro}

To assess the effects of the combination therapy of virotherapy and chemotherapy, 
we use the forward sensitivity index to perform a sensitivity analysis of the virus 
basic reproductive number and the tumour endemic equilibrium  with respect to the 
chemotherapy key parameters.  We now proceed to study the whole model system 
(\ref{eq:uninfected})--(\ref{eq:drug}) with constant drug infusion, that is, $g(t)=q$. 
We first derive the model's basic reproduction number, $R_0$, and calculate its elasticity 
indices with respect to the model parameters to identify which of them are most sensitive 
during tumour infection. Finally, we calculate sensitivity indices of the endemic 
total tumour density with respect to drug infusion.


\subsection{Basic reproduction ratio}
\label{reproductionratio}

A basic reproductive number, in our case, can be defined as the average number 
of new tumour infections generated by one infected cell, via cell lysis, 
during virotherapy in a completely susceptible cell population \cite{van2002reproduction}. 
In general, if $R_0 > 1$, then, on average, the number of new infections resulting from 
one infected cell is greater than one. Thus, viral infections will persist in tumour 
cell populations. If $R_0 < 1$, then, on average, the number of new infections 
generated by one infected cell in virotherapy is less than one. This implies that 
the viral infections will eventually disappear from the cell populations. This threshold 
can as well be used to depict parameters which are most important during tumour infection. 
We calculate $R_0$ using the next- generation matrix approach described in 
\cite{van2002reproduction}. Firstly, we prove that the conditions \textbf{ $A_1-A_5$} 
in Appendix~\ref{appendixC} are satisfied by the model 
Equations (\ref{eq:uninfected})--(\ref{eq:drug}). 

\begin{theorem}
\label{thm:R0-constantinfusion}
The basic reproduction ratio, $R_0$, for constant drug infusion $g(t) = q$, is given by 
\begin{align}
\label{eq:Ro-constantinfusion} R_0 ~= \frac{b\beta \delta U^*(K_c + C^*)}{
\gamma \big[(K_c+ C^*) (\nu_I E_T^* +\delta) + \delta_I C^* \big]\big(K_u + U^* \big)},
\end{align}
where $U^*$, $C^*$ and $E_T ^*$ are all given in Equation (\ref{eq:steadystates1}).  
\end{theorem}

The proof of Theorem~\ref{thm:R0-constantinfusion} can be found in Appendix~\ref{appendixC}. 
Next, we calculate elasticity indices of $R_0$ in response to model parameters 
to ascertain which them are most critical during chemovirotherapy treatment.


\subsection{Sensitivity of the basic reproduction ratio}
\label{sec:Sensitivity_analysis-R0}

\begin{definition}
The Sensitivity and elasticity indices of the basic reproduction ratio, 
$R_0$, with respect to model parameter $p$  are respectively given by 
$S_p = \frac{\partial R_0}{\partial p}$ and 
$e_p = \frac{\partial R_0}{\partial p}\frac{p}{R}$.
\end{definition}

Table~\ref{tab:sensitivityindices} shows the sensitivity and elasticity indices 
obtained with the use of Sage. From the Table one can note that a $1\%$ increase 
in virus burst size, $b$, and infection rate, $\beta$, leads to a  $1\%$ 
increase in $R_0$,  whereas a $1\%$ increase in the virus decay rate, $\gamma$, 
leads to a  1\%  reduction in $R_0$. The indices indicate that $R_0$ is most 
sensitive to virus burst size, infection and decay rates. 
Figure~\ref{fig:elasticity_indices} is a graphical representation of the indices. 
Figure~\ref{fig:tumour-sensitivity} (a) shows the variation of elasticity of $R_0$ 
with respect to drug dosage. The figure indicates that increasing the amount of drug 
infused reduces $R_0$ and thus reduces the number of tumour cells. From the figure 
we can further infer that values of the drug dosage, $q$, from 
40 mg/l to 100 mg/l have a less significant impact on $R_0$. 

\begin{table}
\centering 
\caption{Sensitivity and elasticity indices of $R_0$ with respect to model parameters.}
\label{tab:sensitivityindices}
\begin{tabular}{l|l|l} \hline 
Parameter &Sensitivity index  & Elasticity index \\ 
\hline 
$b$ &  $S_b = 9.88 $  & $e_b = 1$ \\ 

$\beta$ & $S_{\beta} = 988.4$    & $e_{\beta}=1$ \\ 

$\gamma$ & $S_{\gamma} = - 9884$    & $e_{\gamma}=-1$ \\ 

$\delta$ & $S_{\delta} = 2.2325$    & $e_{\delta}= 0.011553$ \\ 

$\alpha$ & $S_{\alpha} = 2.4372$ & $e_{\alpha} = 5.1 \times 10^{-12}$ \\ 

$\delta_U$ & $S_{\delta_U} = -1.004 \times 10^{-7}$ & $e_{\delta_U} = -5.1 \times 10^{-12} $ \\ 

$\delta_I$ & $S_{\delta_I} = -90.32$ & $e_{\delta_I} =- 0.011553$ \\
 
$K_c$ & $ S_{K_c}= 0.2048x$ & $e_{K_c}= 0.0000414$ \\ 

$K_u$ & $S_{K_u}=-1.01\times 10^{-6}$ & $ e_{K_u} = -2.05 \times 10^{-10}$ \\ 

$\psi$ & $S_{\psi} = 0.004$ & $e_{\psi}= 0.000414$ \\ 

$q$ & $S_{q} = -0.008 $ & $ e_{q} = - 0.000414$ \\ \hline 
\end{tabular} 
\end{table}


\subsection{Sensitivity of endemic equilibria}
\label{sec:Sensitivity_analysis-ee}

Sensitivity indices of the endemic equilibria is informative about the establishment 
of the disease. Here, we calculate sensitivity indices of the endemic total tumour 
density with respect to drug infusion. The gist is to determine the relative change 
in the tumour equilibria when the amount of drug infused is changed and thus infer 
the feasible amount of the drug that should be infused.  The sensitivity index 
of the total tumour endemic equilibria $U^* + I^*$ with response to drug infusion $q$ is given by
\begin{align}
\label{eq:sensitivity-of-tumour}
\Gamma^{U^*+ I^*}_q & = \frac{\partial (U^*+I^*)}{\partial q} 
\cdot \frac{q}{U^*+I^*} \nonumber \\
&  = \frac{U^*}{U^*+I^*}\Gamma^{U^*} + \frac{I^*}{U^*+I^*} \Gamma^{I^*}.
\end{align}
where $\Gamma ^{U^*}$ and $\Gamma ^{I^*}$ are calculated from the formula 
in Equation \eqref{eq:derived-sensitivity} derived in Appendix \ref{appendixD}. 
It is worth noting that the endemic equilibrium of the full chemotherapy model 
(\ref{eq:uninfected})--(\ref{eq:drug}) can not explicitly be determined in terms 
of the parameters and therefore we numerically computed the equilibrium values.

We used high values of the tumour reproduction rate and virus burst size, that is 
$\alpha = 0.8$ cells/mm$^3$/day and $b=50$ virions/mm$^3$/day so that $R_0>1$. 
We calculated $\Gamma ^{U^* +I^*}$ for several values of $q$. 
Table~\ref{tab:sensitivity_of_tuomur_equilibria} shows selected indices for different 
amounts of the drug infused with their corresponding basic reproduction ratios. 
Figure~\ref{fig:tumour-sensitivity} (b) is a plot of total tumour sensitivity index 
with respect to drug infusion. The figure suggests that the feasible amount of drug 
that should be infused in a patient lies between 50 mg/l and 100 mg/l. This range 
of values is similar to those inferred from the elasticity indices of $R_0$.

It is worth noting that when $q$ becomes larger than 50 mg/l, the impact of 
chemotherapy on virotherapy ($R_0$ and endemic equilibria) becomes less significant. 
This analysis, nonetheless, does not incorporate the cost for example monetary or 
treatments side effects. These results will further be confirmed using optimal 
control theory. In the next section, we set up an optimal control problem to 
explicitly determine the optimal virus and chemotherapeutic drug dosage 
combination necessary for tumour eradication in body tissue. 

\begin{center}
\begin{table}
\caption{Selected sensitivity indices of the total tumour equilibria, 
$\Gamma ^{U+I} _p$,  in response to drug , $q$ with the corresponding value of $R_0$.}
\label{tab:sensitivity_of_tuomur_equilibria}
\scalebox{0.83}{\begin{tabular}{ c|c|c|c|c|c|c } 
\hline \\
q (mg/l) & 5 & 10 & 15 & 35 & 50 & 100 \\ \hline
$\Gamma ^{U^*+I^*} _p$ & $-8.3\times 10^{-5}$ & $-2.5\times 10^{-5}$ 
& $-1\times 10^{-5}$ & $-6.1\times 10^{-5}$ & $-9.6\times 10^{-6}$ & $-2.4\times 10^{-6}$ \\ 
$R_0$ & $ 51.0476$ & $51.0473$ & $ 51.0470$ & $51.0459$ & $51.0450$ & $ 51.0421$\\ \hline
\end{tabular}}
\end{table}
\end{center}


\section{Optimal control analysis} 
\label{sec:Optimal_control_analysis}

In this section, we propose and analyze an optimal control problem applied to the 
chemovirotherapy model to determine the optimal dosage combination of chemotherapy 
and virotherapy for controlling the tumour. We set the control variables 
$u_1(t)$ and $u_2(t) \equiv g(t)$ to respectively be the supply of viruses from 
an external source of the drug dosage, which are then incorporated into the model 
system's equations (\ref{eq:virus}) and (\ref{eq:drug}) to obtain the following 
control system. For model tractability, we ignore the immune responses. 
\begin{align}
\label{eq:control_system}
\frac{dU}{dt} & = \alpha U\left(1-\frac{U+I}{K}\right)
- \frac{\beta UV}{K_u +U} - \frac{\delta_U  UC}{K_c +C} \nonumber \\
\frac{dI}{dt} & = \frac{\beta UV}{K_u +U}- \delta I -\frac{\delta_I IC}{K_c +C}  \nonumber  \\
\frac{dV}{dt}  & = b\delta I- \frac{\beta UV}{K_u +U}-\gamma V + \bf{u_1(t)}  \nonumber \\
\frac{dC}{dt} & = {\bf{u_2(t)}} -  \psi C
\end{align} 


\subsection{The optimal control problem}
\label{optimal_problem} 

We wish to determine the optimal combination of controls 
$\left(u_1(t),u_2(t)\right)$ that will be adequate to minimize the total 
tumour density ($U(t)+I(t)$) together with the cost of the treatment and negative 
side effects over a fixed time period.  The optimization problem under 
consideration is to minimise the objective functional: 
\begin{align}
\label{eq:objectivefunctional}
J(u_1,u_2)= \int_0^{T_f} \bigg[ U(t)+ I(t) + \left(\frac{A_1}{2}u_1^2(t)
+\frac{A_2}{2} u_2^2(t)\right)\bigg]dt
\end{align}
where $T_f$ is the termination time of the treatment, subject to the control 
system \eqref{eq:control_system}. The two control functions $u_1(t)$ 
and $u_2(t)$  are assumed to be bounded and Lebesgue integrable. 
We thus seek an optimal control pair $(u_1^*,u_2^*)$ such that: 
\begin{align}
\label{eq:minimize}
J(u_1^*,u_2^*) = \min \big\{J(u_1,u_2) | (u_1,u_2)\in \Lambda \big\} 
\end{align} 
where $\Lambda$ is the control set defined by: 
\begin{align}
\Lambda = \big\{(u_1,u_2)|u_i ~\text{is measurable with}  ~0
\leq u_i(t)\leq u_i^{\text{MTD}}, ~t\in [0,T_f], ~i=1,2\big\} \nonumber
\end{align}
where $u_1^{\text{MTD}}$ is the maximum number of viruses body tissue can contain 
and $u_2^{\text{MTD}}$ is the maximum tolerated drug dose. These may as well be viewed 
as the maximum amounts a patient can financially afford. The lower bounds for $u_1$ 
and $u_2$ correspond to no treatment. Here, it is important to note that the balancing 
factors $A_1$ and $A_2$ in the objective functional (\ref{eq:objectivefunctional}) 
are the relative measures of both the cost required to implement each of the associated 
controls as well as the negative sides effects due to the treatment.  


\subsection{Existence of an optimal control pair}

We examine sufficient conditions for the existence 
of a solution to the quadratic optimal control problem. 

\begin{proposition}
\label{prop:5}
There exists an optimal control pair $(u_1^*,u_2^*)$ with a corresponding 
solution ($U^*$, $I^*$, $V^*$, $C^*$) to the model system (\ref{eq:control_system}) 
that minimizes $J(u_1, u_2)$ over $\Lambda$. 
\end{proposition}

\begin{proof}
The proof of Proposition~\ref{prop:5} is based on Theorem 4.1 in Chapter III 
of Fleming and Rishel \cite{fleming2012deterministic}. The necessary conditions 
for existence are stated and verified as follows. 
\begin{enumerate}
\item[(1)] The set of all solutions to the control system (\ref{eq:control_system}) 
and its associated initial conditions and the corresponding 
control functions in $\Lambda$ is non-empty. 

\item[(2)] The control system can be written as a linear function of the 
control variables with coefficients dependent on time and state variables. 

\item[(3)] The integrand in the objective functional in Equation 
\eqref{eq:control_system} is convex on $\Lambda$. 
\end{enumerate}
The right hand sides of the control system \eqref{eq:control_system} are $\mathbb{C}^1$ 
and bounded below and above (see Appendix \ref{appendixA}), thus the solutions to the 
state equations are bounded. It therefore follows from the Picard--Lindel{\"o}f theorem 
\cite{J11} that the system is Lipschitz with respect to the state variables. Thus, 
condition (1) holds. It can be seen from the control system \eqref{eq:control_system} 
that the right hand sides are linearly dependent of $u_1$ and $u_2$. Thus, condition (2) 
also holds. To establish condition (3), we notice that the integrand $\mathcal{L}$ 
in the objective functional \eqref{eq:control_system} is convex 
because it is quadratic in the controls.      
\end{proof}


\subsection{Optimal control characterization} 

In this section, we characterize the optimal controls $u^* = (u_1^*, u_2^*)$ 
which gives the optimal levels for the various control measures and the corresponding 
states ($U^*$, $I^*$, $V^*$, $C^*$). The necessary conditions for the optimal controls 
are obtained using  the Pontryagin's Maximum Principle \cite{pontryagin1987mathematical}. 
This principle converts the model system (\ref{eq:control_system}) into a problem of 
minimizing pointwise a Hamiltonian, $\mathcal{H}$, with respect 
to $u_1$ and $u_2$ as detailed below. 

\begin{proposition}
Let $T=(U,I,V,C)$ and $u=(u_1,u_2)$. If $(T^*(t),u^*(t))$ is an optimal control pair, 
then there exists a non-trivial vector \[\lambda(t) =(\lambda_1(t),\lambda_2(t),
\cdots \lambda_4(t) )\] satisfying the following: 
\begin{align}
\label{eq:char_conditions}
\begin{split}
\lambda_1 ^{\prime}(t) & = -1 -\lambda_1 \bigg[ \alpha \left(1-\frac{2U+I}{K}\right)
- \frac{\beta K_u V }{(K_u +U)^2} - \frac{\delta_U  C}{K_c +C} \bigg],\\ 
& -\lambda_2 \bigg[\frac{\beta K_u V}{(K_u +U)^2}\bigg] 
+ \lambda_3 \bigg[\frac{\beta K_u V}{(K_u +U)^2}\bigg],\\
\lambda_2 ^{\prime}(t) & =-1-\lambda_1 \bigg[\frac{\alpha U}{K} \bigg] 
+ \lambda_2 \bigg[\delta +\frac{\delta_I C}{K_c+C} \bigg]-\lambda_3 b\delta, \\
\lambda_3 ^{\prime}(t) &  = \lambda_1 \bigg[\frac{\beta U}{K_u +U}\bigg] 
-\lambda_2 \bigg[\frac{\beta U}{K_u +U} \bigg] 
+\lambda_3 \bigg[\frac{\beta U}{K_u +U}+\gamma \bigg],\\
\lambda_4 ^{\prime}(t) &  = \lambda_1 \bigg[\frac{K_c \delta_u U}{(K_c +C)^2} \bigg] 
+ \lambda_2 \bigg[\frac{K_c \delta_I I}{(K_c +C)^2} \bigg] +\lambda_4 \psi,\\
\end{split}
\end{align}
with transversality conditions 
\begin{align}\label{eq:transversality-conditions}
\lambda_i (T_f) = 0,\quad  i=1,2,\cdots 4.
\end{align}
and optimal controls:
\begin{align}
\label{eq:conditions:a}
u_1^*  = \min \bigg[u_1^{\text{MTD}},\max \bigg(0,-\frac{\lambda_3}{A_1}\bigg)\bigg], 
\quad u_2^* = \min \bigg[u_2^{\text{MTD}},\max  \bigg(0,-\frac{\lambda_4}{A_2}\bigg)\bigg].
\end{align}
\end{proposition}

\begin{proof}
 The Lagrangian and Hamiltonian for the optimal control 
 system \eqref{eq:control_system} are respectively given by:
\begin{align}
\label{eq:Lagrangian}
\mathcal{L} (U,I,V,C,u_1,u_2) = U +I + \frac{1}{2}\bigg[A_1u_1^2+A_2 u_2^2\bigg] 
\end{align}
and 
\begin{multline}
\label{eq:Hamiltonian}
\mathcal{H} =  U+ I + \frac{1}{2}\bigg[A_1u_1^2+A_2 u_2^2\bigg]  
+ \lambda_1 \bigg[ \alpha U\left(1-\frac{U+I}{K}\right)- \frac{\beta UV}{K_u +U} 
- \frac{\delta_U  UC}{K_c +C} \bigg]\\ 
+ \lambda_2 \bigg[ \frac{\beta UV}{K_u +U}- \delta I -\frac{\delta_I IC}{K_c +C} \bigg] 
+ \lambda_3 \bigg[b\delta I- \frac{\beta UV}{K_u +U}-\gamma V + u_1 \bigg] 
+ \lambda_4\big[u_2 - \psi C \big].
\end{multline}
We thus obtain Equation \eqref{eq:char_conditions} 
using the Pontryagin's maximum principle from
\begin{align}
\label{eq:conditions:b}
\begin{split}
\lambda ^{\prime}(t) 
= -\frac{\partial \mathcal{H}}{\partial T}\bigg(t,T^*(t),u^*(t), \lambda(t) \bigg).
\end{split}
\end{align}
The transversality conditions are as given in Equation \eqref{eq:transversality-conditions} 
since all states are free at $T_f$. The Hamiltonian is maximized with respect to the controls 
at the optimal control $u^* = (u_1^*, u_2^*)$. Therefore $\mathcal{H}$ is differentiated 
with respect to $u_1$ and $u_2$ on $\Lambda$, respectively, to obtain 
\begin{align*}
0=\frac{\partial \mathcal{H }}{\partial u_1} & = A_1 u_1 + \lambda_3,\\
0=\frac{\partial \mathcal{H}}{\partial u_2} & = A_2 u_2 +\lambda_4. 
\end{align*}
Thus, solving for $u_1^*$ and $u_2^*$ on the interior sets gives
\[
u_1^* = -\frac{\lambda_3}{A_1}, \quad u_2^* = -\frac{\lambda_4}{A_2}.
\]
By standard control arguments involving the bounds on the controls, 
$0\leq u_1(t)\leq u_1^{\text{MTD}},~~0\leq u_2(t)\leq u_2^{\text{MTD}}$, 
we conclude that:
\begin{align}
\label{eq:OCconditions}
u_1^*  = \min \bigg[u_1^{\text{MTD}},\max \bigg(0,-\frac{\lambda_3}{A_1}\bigg)\bigg], 
\quad u_2^* = \min \bigg[u_2^{\text{MTD}},\max  \bigg(0,-\frac{\lambda_4}{A_2}\bigg)\bigg].
\end{align}
\end{proof}


\subsection{The optimality system}
\label{OCSystem}

In summary, the optimality system consists of the control system 
(\ref{eq:control_system}) and the adjoint system \eqref{eq:char_conditions} 
with its transversality conditions \ref{eq:transversality-conditions}, 
coupled with the  control characterizations \eqref{eq:OCconditions}. Next, 
we proceed to solve numerically the proposed model and the optimal control problem. 


\section{Numerical simulations}
\label{sec:Numerical_simulations}

In this section, we discuss the numerical solutions of both the chemovirotherapy 
model equations (\ref{eq:uninfected})--(\ref{eq:drug}) and the optimal control 
problem defined in Section~\ref{OCSystem}. We also outline the parameter choices 
and the initial conditions. We use parameter values in Table~\ref{tab:parametervalues1} 
to solve model equations and the optimality system. The numerical solutions of the model 
equations are illustrated using MATLAB, while the optimality system 
was solved using a fourth order Runge--Kutta iterative method.


\subsection{Parameter values and initial conditions}

Some of the parameter values were obtained from fitted experimental data for untreated 
tumours and virotherapy in mice \cite{ZTK08} and others from biological facts in the literature. 
A tumour nodule can contain about $10^5-10^9$ tumour cells \cite{SS95}. The carrying capacity 
is therefore considered to be $10^6$ cells per $mm^{3}$. \textit{In vivo} experiments estimate 
the intrinsic rate of growth to be $0.1-0.8$ day$^{-1}$ \cite{BL14}. We consider the number 
of viruses produced per day, $b$, to be in the range $0-1000$ virions \cite{BD90}. The amount 
of drug infused in body tissue, $q$, is considered to be $5$ milligrams per day and the decay rate, 
$\psi$, to be $4.17$ milligrams per day, values which conform to cancer pharmacokinetic studies 
\cite{GS11,RGG11}. Since infected tumour cells multiplication is enhanced by the oncolytic 
virus replication, the tumour cells lysis, $\delta_I$, is considered to be greater than that 
for uninfected tumour cells $\delta_U $ (see Ref \cite{SDP}). We considered a virus-specific 
immune response rate, $\phi$, of $0.7$ day$^{-1}$ \cite{bollard2016t}. Both virus and tumour 
specific immune decay rates are assumed to be $0.01$ day$^{-1}$, given the fact that 
their lifespan is less than 100 days \cite{le2014mathematical,nayar2015extending}. 
The tumour-specific production rate was estimated at $0.5$ cells per mm$^3$ 
per day \cite{gajewski2013innate,le2014mathematical}. 
In all our simulations, unless stated otherwise, we considered the initial concentrations 
$U_0=10000$ cells per mm$^3$, $I_0=100$ cells per mm$^3$, $V_0=500$ virions per mm$^3$, 
$E_{v_0}=100$ cells per mm$^3$, $E_{T_0}=100$ cells per mm$^3$ and $C_0=100$ g/ml 
with a high percentage of untreated tumour cell count to require treatment.


\subsection{Chemovirotherapy model}

\begin{figure}
\centering 
\includegraphics[trim=0.5cm 5cm 0.5cm 3cm,clip=true,width=13cm]{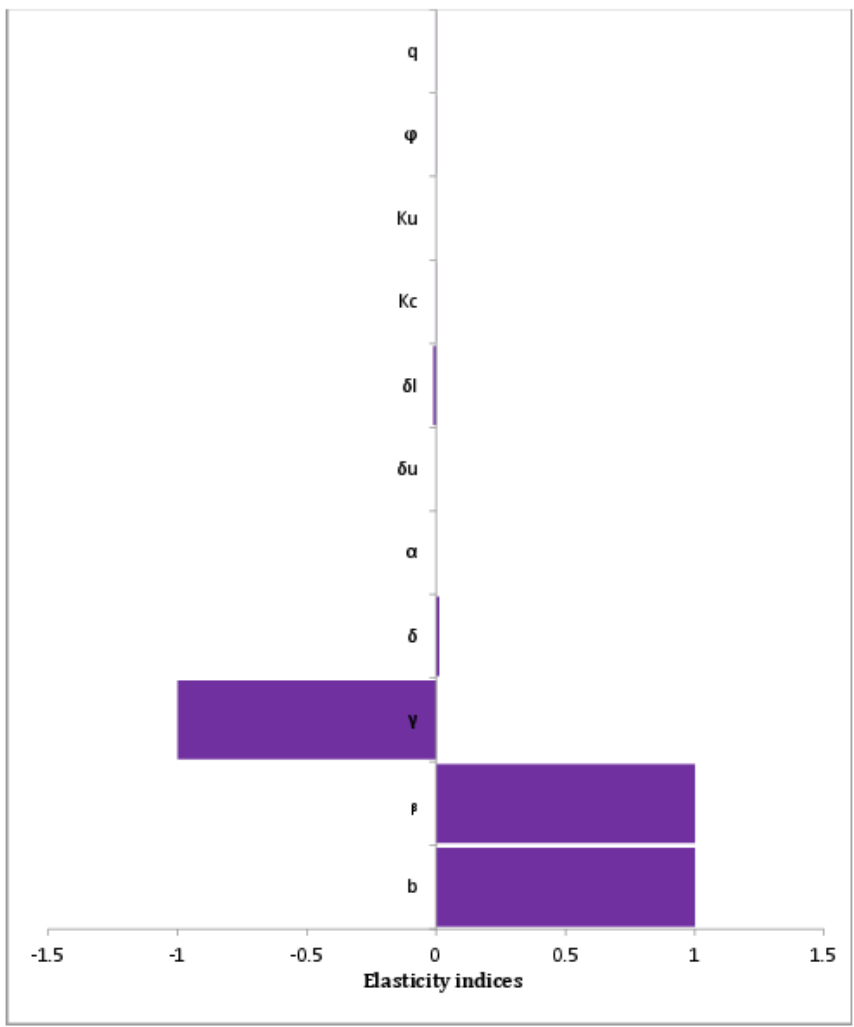}
\caption{Visual representation of the elasticity indices of $R_0$ with respect 
to model parameters. The bar graph shows that the virus burst size, infection and decay 
rates $b$, $\beta$ \& $\gamma$ have the highest elasticity indices with virus 
decay being negatively correlated to $R_0$.}
\label{fig:elasticity_indices}
\end{figure}

\begin{figure}
\centering 
\begin{tabular}{c}
(a)\includegraphics[width=12.5cm,height=10cm]{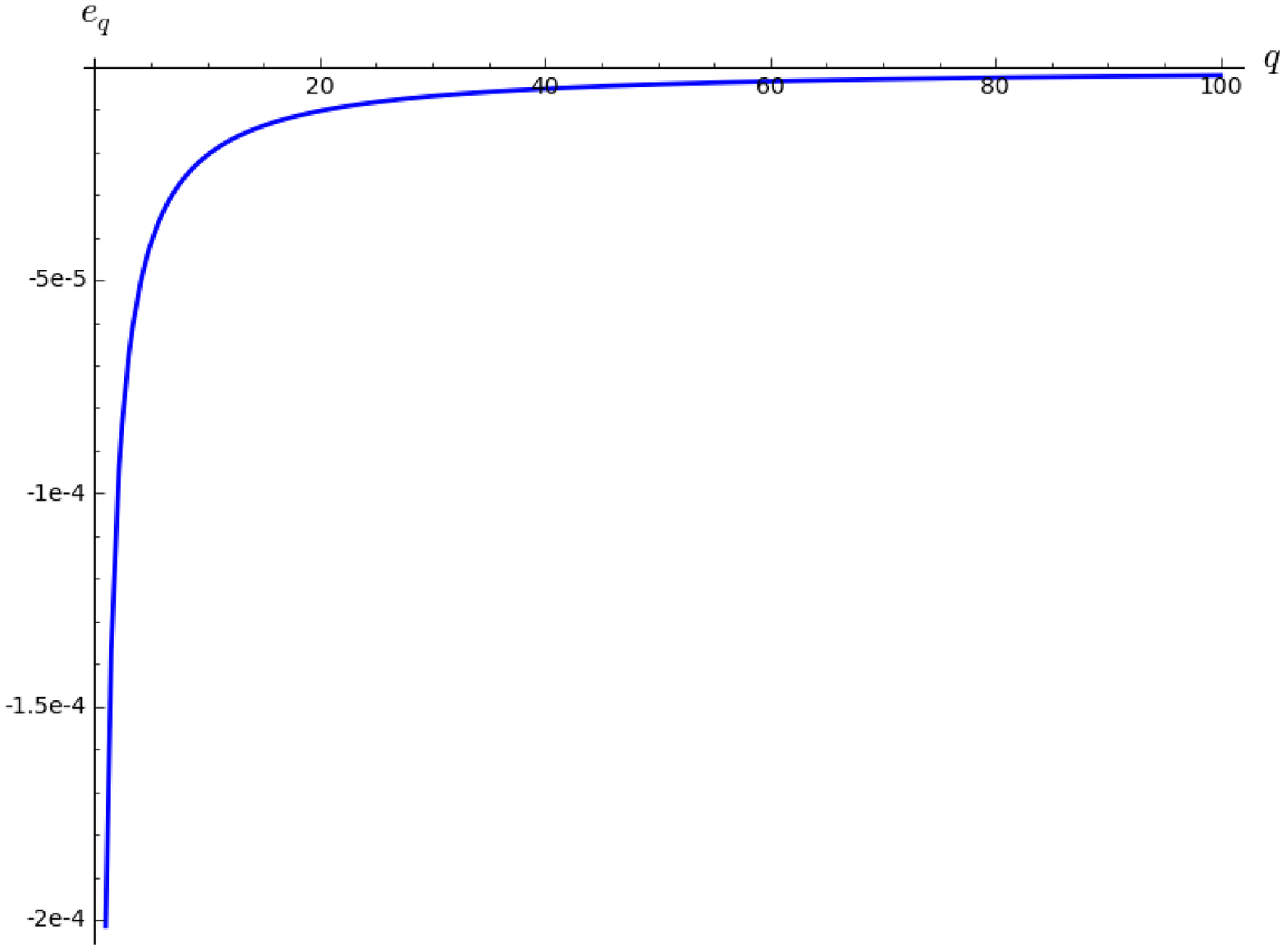}\\
(b)\includegraphics[width=12.5cm,height=10cm]{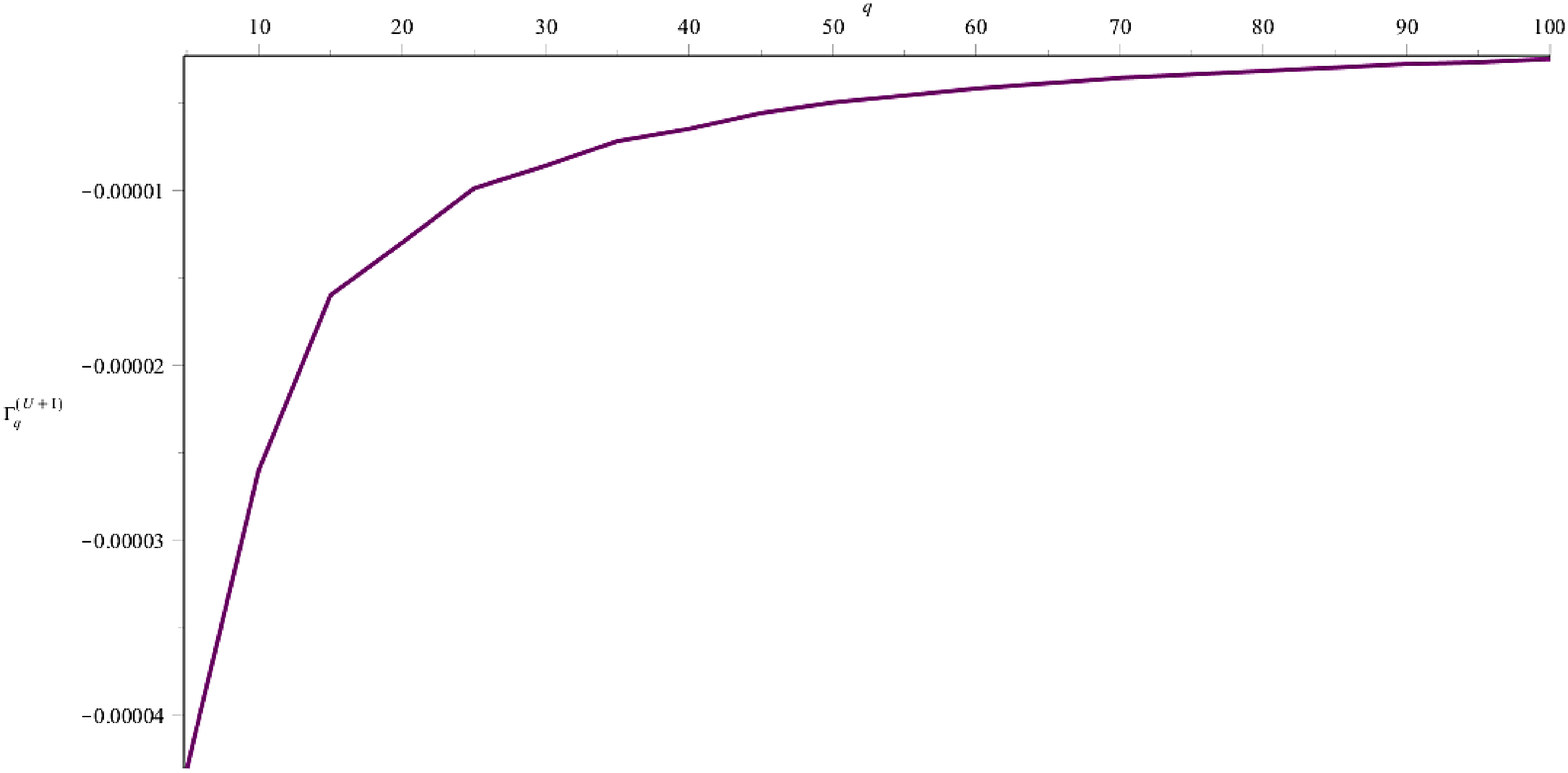}
\end{tabular}
\caption{Plots of the elasticity indices, $e_q$, and $\Gamma^{U^*+ I^*}_q$ 
against the drug dosage $q$. Both Figures (a) and (b) depict that increasing the amount 
of drug infused reduces viral multiplication thus reducing the sensitivity indices. 
The figures further suggest that values of $q$ from 40 to 100 mg/l have 
minimal negative impact on viral replication.}
\label{fig:tumour-sensitivity}
\end{figure}

\begin{figure}
\centering 
\begin{tabular}{c}
a \includegraphics[scale=0.87]{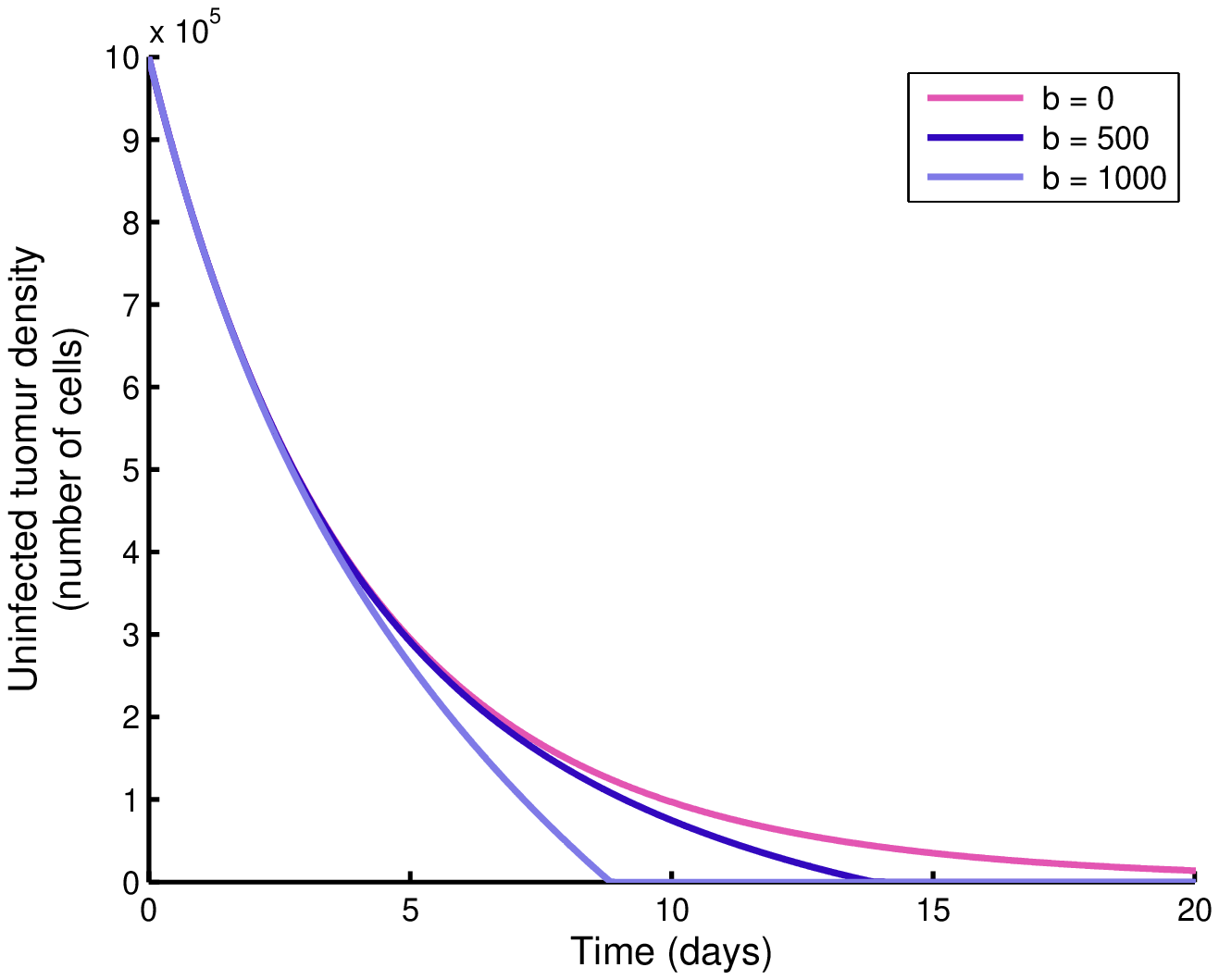}\\
b \includegraphics[scale=0.87]{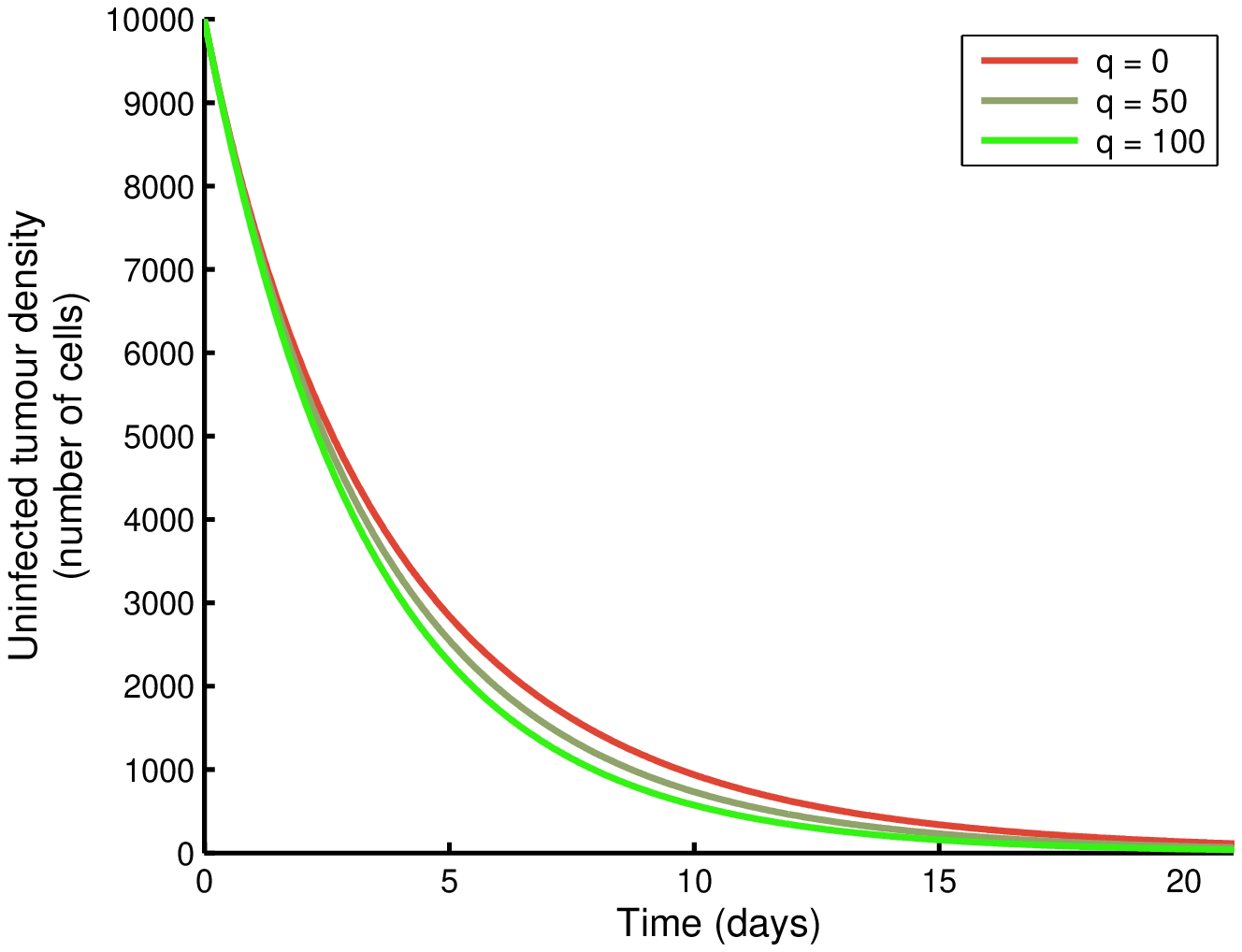}
\end{tabular}
\caption{(a) Plots of uninfected cell density with constant drug infusion 
and for different virus burst sizes. The plot shows that an increase in the virus 
burst size reduces the tumour density. (b) Plots of uninfected cell density 
with constant drug infusion and for different drug infusion rates. The plot shows 
that an increase in the drug infusion rate reduces the tumour density 
by a relatively small magnitude.}
\label{fig:constant_uninfected_varyingbq}
\end{figure}

\begin{figure}
\centering 
\begin{tabular}{c}
a \includegraphics[scale=0.7]{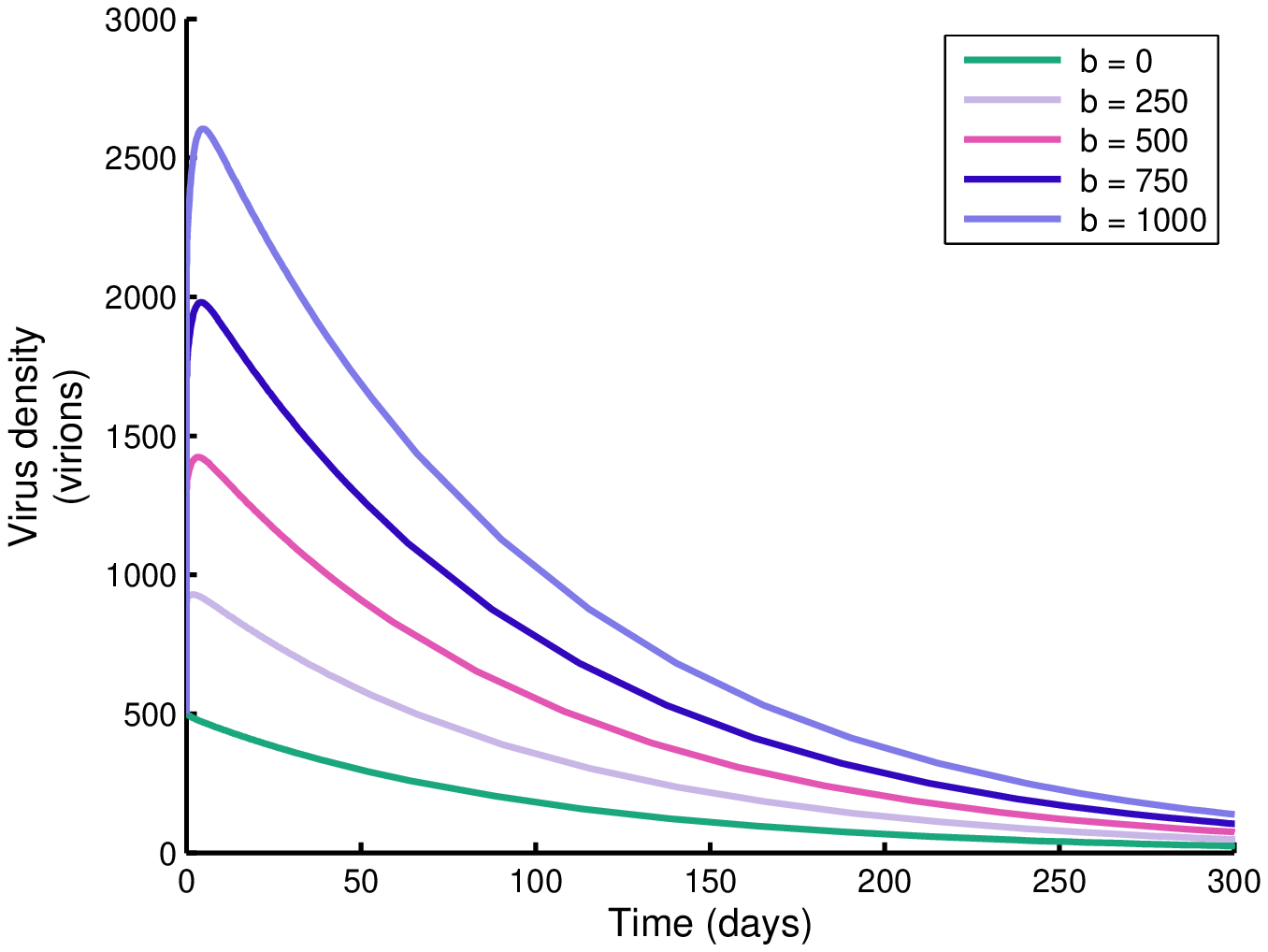}\\
b \includegraphics[scale=0.7]{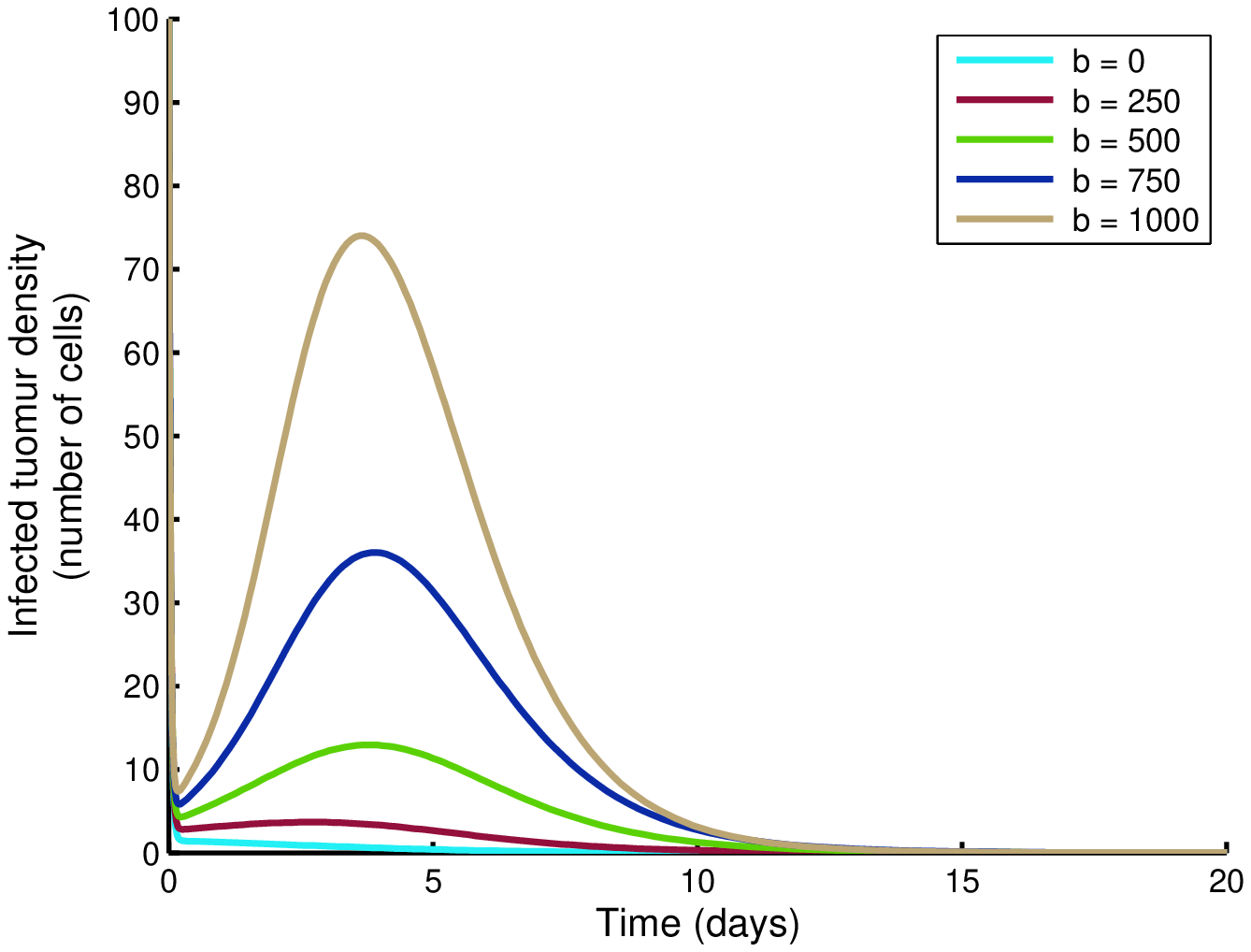}
\end{tabular}
\caption{Plots of the virus and infected tumour densities for different 
virus burst size. The plots show that both densities increase 
with increasing virus burst size.}
\label{fig:constant_virus_varyingb}
\end{figure}

\begin{figure}
\centering 
\begin{tabular}{c}
a \includegraphics[scale=0.7]{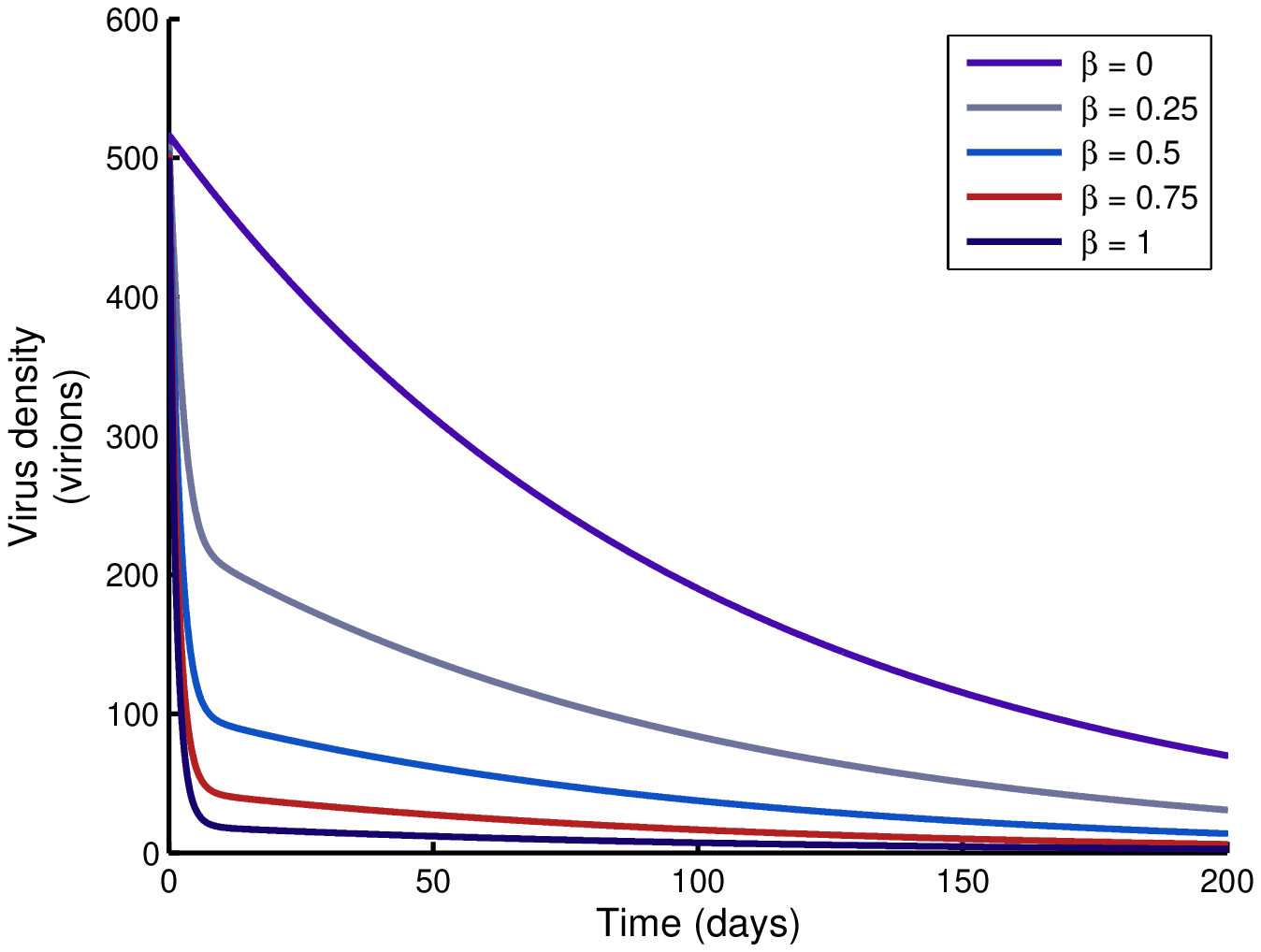}\\
b \includegraphics[scale=0.7]{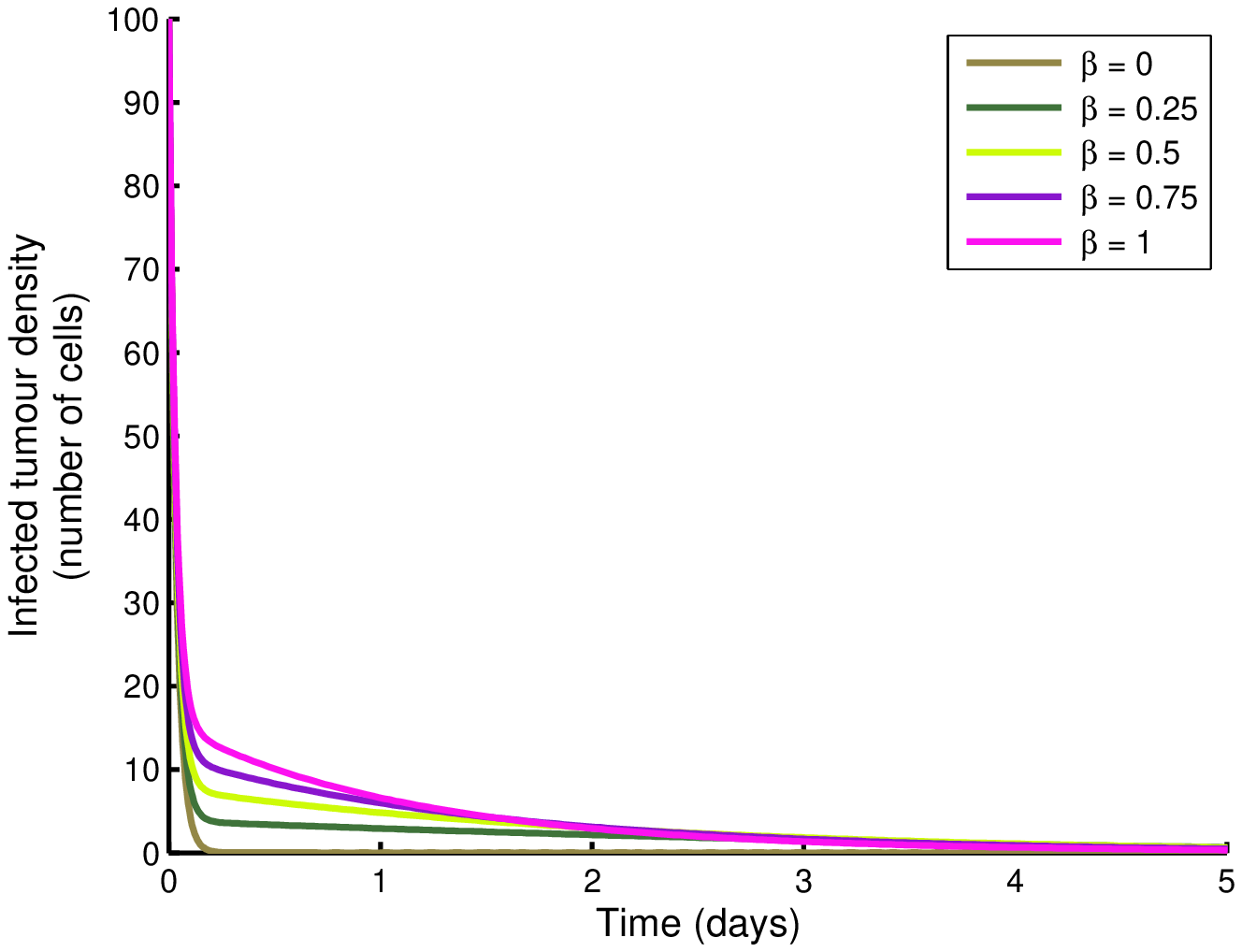}
\end{tabular}
\caption{(a) Plots of the virus density for different virus infection rate. 
The plots show the virus density reduces with increasing virus infection rates.
(b) Plots of infected tumour density for different infection rates. 
The figure shows that the infected tumour density increases as the infection rate increases.}
\label{fig:constant_virus_varyingbeta}
\end{figure}

Figures~\ref{fig:constant_uninfected_varyingbq} (a)\&(b) are respectively plots 
of the uninfected tumour density for different virus and drug doses. Both figures 
depict that an increase in the virus burst size ($b$) and drug dose ($q$) reduces 
the tumour density. In Figure \ref{fig:constant_uninfected_varyingbq} 
(a), without virotherapy, that is $b = 0$ virions, the tumour is depleted in about 
10 days whereas without chemotherapy, that is $q=0$ milligrams per day, it takes about 20 days.

Figures~\ref{fig:constant_virus_varyingb} and \ref{fig:constant_virus_varyingbeta} (a)\&(b) 
respectively show the effect of virus burst size on the virus and infected tumour cell densities. 
Both figures show that the densities with increasing burst size and infection rate. 
From Figure, for example \ref{fig:constant_virus_varyingb}, an increase of $b$ 
from 250 to 500 virions led to an increase of the virus number to about 
$700$ $virions$ from $500$ $virions$, in the first 100 days. 
 
Figure~\ref{fig:constant_virus_varyingbeta} (a)\&(b) show that increasing virus infection rate, 
$\beta$, reduces the virus density and increases the infected tumour density. For example 
from Figure~\ref{fig:constant_virus_varyingbeta} (a), an increase in $\beta$ from 0 to 1 
cells per mm$^3$ per day led to a reduction by about 100 viruses from 120 whereas 
it led to an increase by about 10 viruses in the first day, from 0 as seen 
in Figure~\ref{fig:constant_virus_varyingbeta} (b). 
 

\subsection{Optimal control problem}	

Figure~\ref{fig:optimal-tumour} is a plot of total tumour density of the optimal 
control solution to the problem formulated in Section~\ref{sec:Optimal_control_analysis}. 
It shows the tumour density being reduced by the combinational therapy 
treatment to a very low state in less than a week. 
	
\begin{figure}
\centering 
 \includegraphics[scale=0.7]{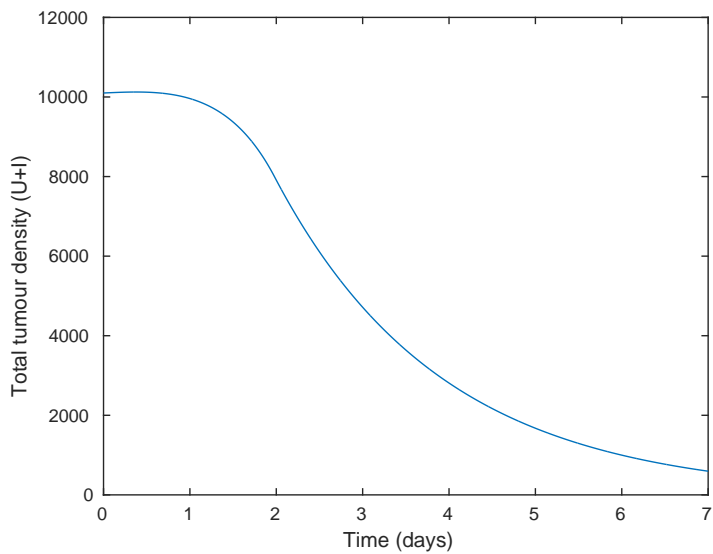}
\caption{Total tumour density with optimal control. 
The tumour density is reduced in a very short time period.}
\label{fig:optimal-tumour}
\end{figure}

Figures~\ref{fig:optimal-u_1} (a)\&(b) represent the optimal controls $u_1$ and $u_2$. 
The figures suggest that the optimal values of the number of viruses and drug concentration 
that yields efficient combination therapy are half their maximal 
dosages $u_1^{\text{MTD}}$ and $u_2^{\text{MTD}}$. 

\begin{figure}
\centering 
\begin{tabular}{c}
(i)\includegraphics[scale=0.7]{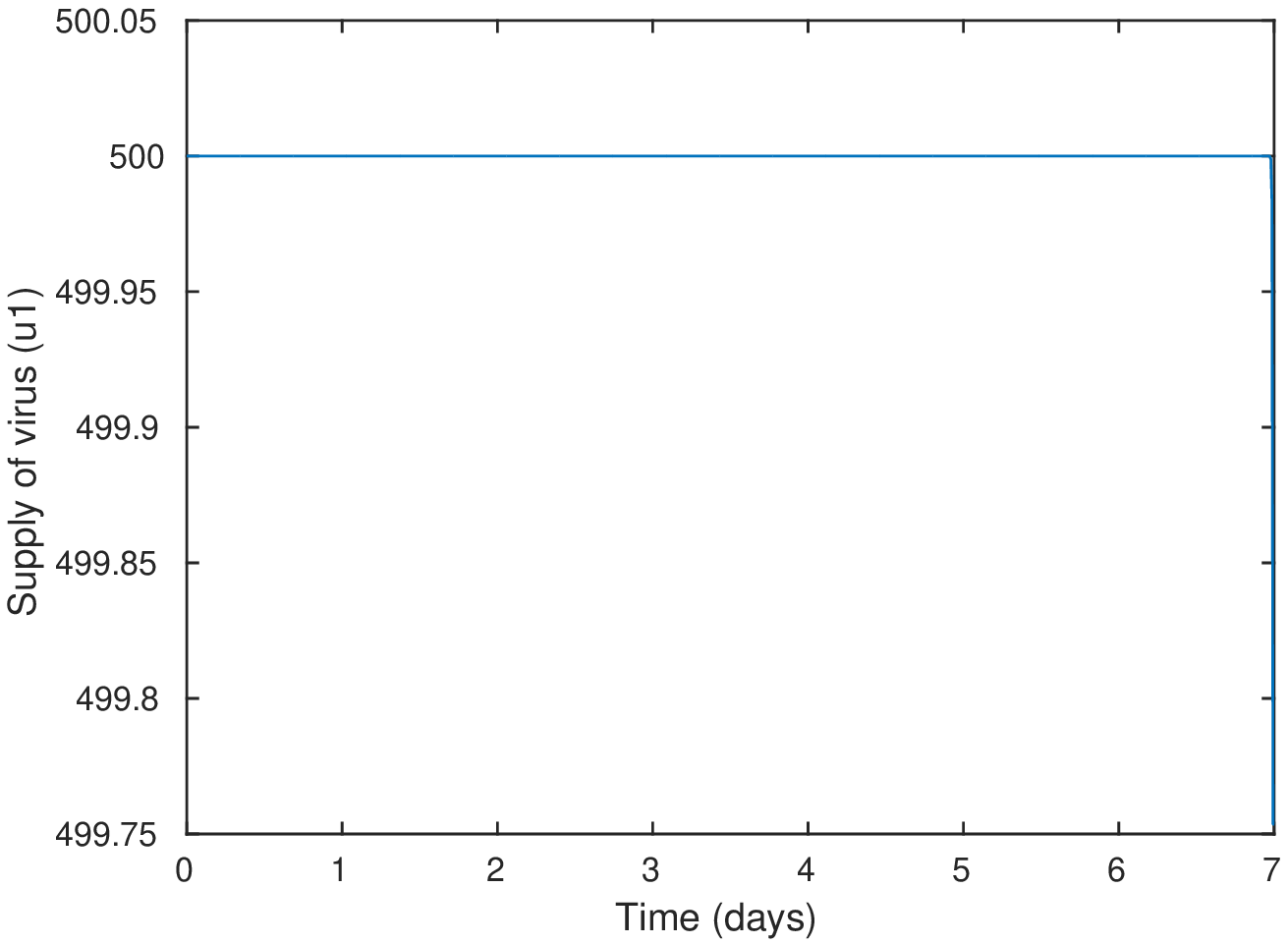}\\
(ii)\includegraphics[scale=0.7]{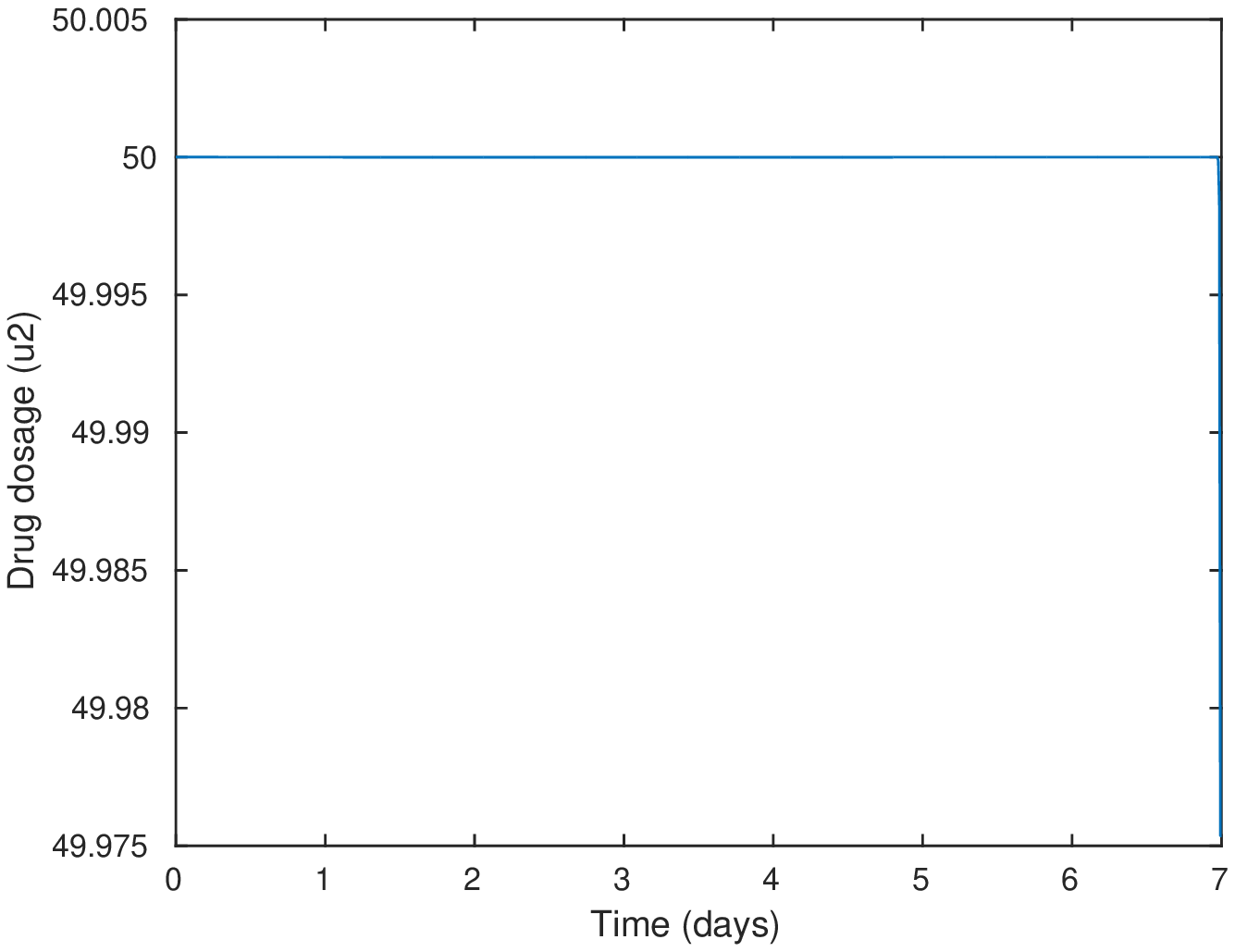}
\end{tabular}	
\caption{Optimal control variable variables $u_1$: The external supply 
of viruses and  $u_2$: Drug dosage.  The Figures depict 500 virions as the optimal 
number of viruses and the optimal drug dosage to be 50 m/g per.}
\label{fig:optimal-u_1}
\end{figure}


\section{Conclusion}
\label{conclusion}

Successful cancer treatment often requires a combination of treatment regimens. 
Recently, combined oncolytic virotherapy and chemotherapy has been emerging 
as a promising, effective and synergetic  cancer treatment.  In this paper 
we consider a combination therapy with chemotherapy and virotherapy 
to investigate how virotherapy could enhance chemotherapy. 
To this end,  we developed a mathematical model and an optimal control problem 
to determine the optimal chemo-virus combination to eradicate a tumour. 
We firstly validated the model's plausibility by proving existence, 
positivity and boundedness of the solutions. 
We analysed the model in four scenarios: without treatment, with chemotherapy, 
with oncolytic virotherapy alone, and with combined chemotherapeutic drug and virus therapy. 
A basic reproduction number for the infection of the tumour cells was calculated 
to analyse the model's tumour endemic equilibrium. Furthermore, sensitivity 
and elasticity indices of the basic reproduction number and tumour endemic 
equilibria with respect to drug infusion were calculated. The optimal control 
problem was solved using the Pontryagin's Maximum Principle. Numerical solutions 
to the proposed model were carried out to illustrate the analysis results. 
Similarly, numerical solutions to the optimal control problem were carried out to identify 
the optimal dosages for the minimization of the chemo-virus combination. 
Model analysis and simulations suggest the following  results.

The stability analysis showed that a tumour can grow to its maximum size 
in a case where there is no treatment. It also demonstrated that chemotherapy 
alone is capable of clearing tumour cells provided that the drug efficacy 
is greater than the intrinsic tumour growth rate. This can be evidenced 
in a study by Dasari and Tchounwou \cite{dasari2014cisplatin} and references therein.
Furthermore, the result emphasised that if the chemotherapeutic drug's efficacy 
is high enough but with a toxic drug then the tumour cannot be eliminated. If, however, 
the drug is not toxic, then a dosage which is not larger than the maximum tolerated 
dose can be given while still allowing for a stability condition of the tumour 
free equilibrium, thus eliminating the tumour. 

In the case of virotherapy alone, stability analysis revealed that virotherapy 
on its own may not clear tumour cells  but rather highly enhances chemotherapy 
in destroying tumour cells. Several clinical and experimental studies for example 
\cite{PJ10,binz2015chemovirotherapy,VJ02,relph2016cancer} confirm this result. 
In a review of clinical studies by Binz and Ulrich \cite{binz2015chemovirotherapy}, 
it is stated that phase II/III clinical trials on combining adenovirus H101 with 
a chemotherapeutic drug, cisplatin and 5-fluorouracil, leads to a 40\% improvement 
compared to chemotherapy alone for the treatment of patients with head and neck cancer. 
Adenovirus Ad-H101 was consequently approved for the treatment of head
and neck cancers in China and recently in the United states after phase III clinical 
trials showed a 72--79\% response rate for the combination of the virus with 
5-fluorouracil compared to a 40\% with chemotherapy alone  \cite{garber2006china}.  

Sensitivity analysis and elasticity indices of the basic reproduction ratio indicated 
that successful chemovirotherapy is highly dependent on the virus burst size and 
its infection rate as well as on the drug infusion and its decay rates. Larger virus 
burst size and higher replication rates lead to a lower tumour concentration. Whereas, 
low decay and high doses of the drug that the body can tolerate lead to better treatment 
results. Sensitivity analysis of both the basic reproduction ratio and model equilibria 
suggested the feasible drug infusion to be between 50 and 100 milligrams per litre. 
The effect of virus burst size and infection rates in determining the success 
of virotherapy have extensively been confirmed in recent mathematical studies 
(see for example \cite{malinzi2017modelling,JP11,crivelli2012mathematical}). 
Nonetheless, the right dose of chemotherapy treatment has always been a question 
and a concern to both clinicians and mathematicians alike 
\cite{konstorum2017addressing,frei1980dose}. In this study, optimal control 
results further indicated that $50\%$ of the maximum drug and virus 
tolerated dosages (MTDs) optimize chemovirotherapy. 

In addition, numerical simulations suggested that with the use of both chemotherapy 
and virotherapy, a tumour may be eradicated in less than a month. We know that this 
is not realistic for human cancers and is strictly based on the experimental data used 
in the numerics. Simulations further showed that a successful combinational therapy 
of cancer drugs and viruses is mostly dependent on virus burst size, infection rate 
and the amount of drugs supplied into a patient's body which is in agreement 
with recent studies \cite{malinzi2017modelling,JP11,crivelli2012mathematical}. 

The mathematical model we developed in this study is a simple one, for example, 
it considers the cell densities to only be time dependent. It is therefore imperative 
for further studies to incorporate more biological aspects like spatial variation 
of the cell concentrations. Another facet would be to investigate the effect of toxicity 
of both viruses and the drug on normal body tissue. 
The key to improving combined virotherapy and chemotherapy lies in quantifying 
the dependence of treatment outcome on immune stimulation. Another extension 
to this model, would be to include other subpopulations of the immune system. 
Nevertheless, the results here emphasise the treatment characteristics that are 
vital in combining drugs and viruses to treat cancer and an optimal drug 
and virus dosage is suggested.  


\begin{appendices}

\section{Appendix A}
\label{appendixA} ~~~

\medskip

\noindent A.1.  Proof of Theorem~\ref{thm:welposedness}(i) 
on the existence and uniqueness of model solutions

\begin{proof}
The functions on the right hand side of Equations (\ref{eq:uninfected})--(\ref{eq:drug}) 
are $\mathbb{C}^1 $ on $\mathbb{R}^n$. Thus, it follows from the Picard--Lindel{\"o}f 
theorem (see \cite{J11}) that (\ref{eq:uninfected})--(\ref{eq:drug}) exhibit a unique solution. 
\end{proof}

\noindent A.2.  Proof of Theorem~\ref{thm:welposedness}(ii)\&(iii) on positivity and boundedness

\begin{proof}
\label{proof1}
The solution to Equation (\ref{eq:drug}) is given by 
\[
C(t)=\exp(-\psi t)\left(\int_{0}^t g(s)\exp(\psi s)ds + C(0)\right) \geq 0.
\]
Let the model equations (\ref{eq:uninfected})-(\ref{eq:immune_T}) be written 
in the form $x^{\prime} = F(t,x)$. The functions $F(x,t)$ on the right hand side 
of Equations (\ref{eq:uninfected})--(\ref{eq:immune_T}) have the property of 
\[
F(U,I,V,E_V,E_T,t)\geq 0~\text{whenever}~ x\in [0,\infty )^n,~ x_j 
= 0,~t\geq 0, 
\]
where $x = (U,I,V,E_V,E_T)$. Thus, it follows from 
Proposition~A.1 in \cite{thieme2003mathematics} that 
\[
x(t)\in [0,\infty)^n ~\text{for all} ~t\geq t_0 \geq 0~~\text{whenever}~~ x(t_0) \geq 0
\]
and thus the model solutions $x=(U,I,V,E_V,E_T)$ are positive 
and bounded for positive initial values $x^0$. 
\end{proof}

\noindent A.3. Proof of Theorem~\ref{thm:welposedness}(iv) on positive invariance

\begin{proof}
\label{proof3}
The proof directly follows from proofs of Theorems~\ref{thm:welposedness} (i)--(iii).
\end{proof}


\section{Appendix B}
\label{appendixB} ~~~

\noindent B.1. Proof of stability of tumour endemic state, 
$X_2$, in Proposition~\ref{prop:1}. 

\begin{proof}
Denote $U^* : = U_p +U_r$ where 
\begin{align*}
U_p & = -\frac{b}{2a} = \frac{K}{2\alpha}\bigg(\alpha 
-  \frac{\alpha \kappa}{K}-M \bigg),\\
U_r & = \frac{\sqrt{b^2-4ac}}{2a} =  \frac{K}{2\alpha}\sqrt{\bigg(\alpha 
- \frac{\alpha \kappa}{K}-M \bigg)^2 +4\frac{\alpha^2\kappa}{K}}.
\end{align*}
The characteristic polynomial of the Jacobian matrix evaluated at $X_2$ is given by 
\begin{align}
\label{eq:char1:b}
P(x)  = x^2 + P_1x + P_0
\end{align}
where 
$P_0 = P_{01}U^* + P_{00}$ and $P_1 = P_{11}U^* + P_{10}$, 
\begin{align*}
P_{00} & =   K\kappa\, \left( \alpha\,\delta_{T}\,K-\beta_{T}\,\nu_{u}\,K+\alpha\,
\kappa\,\delta_{T} \right), \\
P_{01} & = {\frac {K \big( K(\alpha \delta_T - \beta_T\nu_U)^2+\kappa\,{\delta_{T}}
^{2}{\alpha}^{2}+\alpha\,\beta_{T}\,\kappa\,\nu_{u}\,\delta_{T}
\big) }{\alpha\,\delta_{T}}}, \\
P_{10} & =   K\kappa\,\delta_{T}\, \left( \alpha+\delta_{T} \right),  \\
P_{11} & =  K \left( \alpha\,\delta_{T}-\beta_{T}\,\nu_{u}+{\delta_{T}}^{2} 
\right). 
\end{align*} 
Using  Routh--Hurwitz criterion, the endemic steady state, $X_2$, 
is locally asymptotically stable if $P_ 1>0$ and $P_0>0$. 
	
We notice that $P_{01}>0$. This means that $P_0 >0$ 
if and only if $U^* > -P_{00}/P_{01}$ which is equivalent to 
\begin{align*}
U_r & > -\frac{P_{00}}{P_{01}}-U_p \Leftrightarrow U_r ^2 
- \left( \frac{P_{00}}{P_{01}} +U_p\right)^2 > 0.	 
\end{align*}
With the help of Maple$^{\tiny{\textregistered}}$, we find that 
\begin{align*}
U_r^2 - \left( \frac{P_{00}}{P_{01}} +U_p\right)^2 
& = \alpha\,\beta_{T}\,{\kappa}^{2}\nu_{u}\,\delta_{T}\, 
\bigg( K^2 (\alpha \delta_T - \beta_T\nu_U)^2
+2\,K\kappa\,{\delta_{T}}^{2}{\alpha}^{2}\\
& \quad \quad +2\,K\alpha\,\beta_{T}\,\kappa\,\nu_{u}\,\delta_{T}+{
\alpha}^{2}{\kappa}^{2}{\delta_{T}}^{2} \bigg) > 0. 
\end{align*}
Therefore $P_0 >0$.
	
For $P_1$, we discuss the following cases: 
\begin{enumerate}
\item[(i)] If $P_{11}> 0$ then $P_1>0$. 
\item[(ii)] If $P_{11}<0$ then $P_1>0$ if and only if 
$U^* < -P_{10}/P_{11}$ which is equivalent to 
\begin{align*}
U_r^2 - \left(\frac{P_{10}}{P_{11}}  + U_p \right)^2 < 0.
\end{align*}
Using Maple$^{\tiny{\textregistered}}$, we find that 
\[
U_r^2 - \left(\frac{P_{10}}{P_{11}} +U_p \right)^2 
= \beta_{T}\,\kappa\,\nu_{u}\,\delta_{T}\, \bigg( K(\alpha\,\delta_{T}\,
- \beta_{T}\,\nu_{u}\,+{\delta_{T}}^{2})-\alpha \kappa(\alpha+\delta_{T}) \bigg) 
<0.
\]
This implies that $P_1>0$.
\end{enumerate}
Therefore the endemic steady state, $X_2$, is locally asymptotically stable. 
\end{proof}

\noindent B.2. Proof of stability of endemic state, $C_2$, in Proposition~\ref{prop:2}

\begin{proof} 
The characteristic polynomial, $P(x)$, of the Jacobian evaluated at $C_2$ is given by 
{\small 
\begin{align} 
\label{eq:chapolyc2}
P(x) & = \newcommand{\Bold}[1]{\mathbf{#1}}x^{3} + \left(\frac{2 \, U \alpha}{K} 
- \alpha + \delta_{T} + \frac{C \delta_{U}}{C + K_{c}} + \frac{U \beta_{T} 
\nu_{U}}{{\left(U + \kappa\right)} \delta_{T}} + \psi\right) x^{2} \\
&  +\bigg(\frac{2 \, U \alpha \delta_{T}}{K} - \alpha \delta_{T} 
+ \frac{C \delta_{T} \delta_{U}}{C + K_{c}} + \frac{2 \, U \beta_{T} \nu_{U}}{U + \kappa} 
- \frac{U^{2} \beta_{T} \nu_{U}}{{\left(U + \kappa\right)}^{2}}  \\
& + \frac{2 \, U \alpha \psi}{K} - \alpha \psi + \delta_{T} \psi 
+ \frac{C \delta_{U} \psi}{C + K_{c}} + \frac{U \beta_{T} \nu_{U} 
\psi}{{\left(U + \kappa\right)} \delta_{T}}\bigg) x \\
&+ \frac{2 \, U \alpha \delta_{T} \psi}{K} - \alpha \delta_{T} \psi 
+ \frac{C \delta_{T} \delta_{U} \psi}{C + K_{c}} + \frac{2 \, U \beta_{T} 
\nu_{U} \psi}{U + \kappa} - \frac{U^{2} \beta_{T} \nu_{U} \psi}{{\left(U + \kappa\right)}^{2}}=0. 
\end{align}}
The conditions for stability of the endemic state, $C_2$, therefore follow 
from the Routh--Hurwitz criterion and are stated as: 
\begin{align}
\label{eq:conditions-forstabilityofc2}
a_2: & = \frac{2 \, U^* \alpha}{K} - \alpha + \delta_{T} 
+ \frac{C^* \delta_{U}}{C^* + K_{c}} + \frac{U^* \beta_{T} \nu_{U}}{{\left(U^* 
+ \kappa\right)} \delta_{T}} + \psi > 0, \nonumber \\
a_1 :& = \frac{2 \, U^* \alpha \delta_{T}}{K} - \alpha \delta_{T} 
+ \frac{C^* \delta_{T} \delta_{U}}{C^* + K_{c}} + \frac{2 \, U^* \beta_{T} \nu_{U}}{U^* 
+ \kappa} - \frac{U^{*^2} \beta_{T} \nu_{U}}{{\left(U^* + \kappa\right)}^{2}}  \nonumber \\
& \quad \quad \quad + \frac{2 \, U \alpha \psi}{K} - \alpha \psi + \delta_{T} \psi 
+ \frac{C^* \delta_{U} \psi}{C^* + K_{c}} + \frac{U^* \beta_{T} \nu_{U} \psi}{{\left(U^* 
+ \kappa\right)} \delta_{T}} > 0, \nonumber \\
& \hspace{-3em} a_0:  = \frac{2 \, U^* \alpha \delta_{T} \psi}{K} - \alpha \delta_{T} \psi 
+ \frac{C^* \delta_{T} \delta_{U} \psi}{C^* + K_{c}} + \frac{2 \, U^* \beta_{T} \nu_{U} \psi}{U^* 
+ \kappa} - \frac{U^{*^2} \beta_{T} \nu_{U} \psi}{{\left(U^* 
+ \kappa\right)}^{2}} > 0 ~ \text{and} ~ a_2a_1>a_0.
\end{align}
\end{proof}


\section{Appendix C} 
\label{appendixC} ~~~

\medskip

\noindent C.1. Proof of  conditions \textbf{ $A_1 -A_5$} for the derivation of $R_0$

\begin{proof}
Denote $x=(x_1,x_2,\cdots x_5)^T = (U,I,V,E_V,E_T)^T$ with $x_i \geq 0$ be the set 
of cell compartments divided into infected cells $i=1,\cdots,m$ and uninfected 
cells $i=m+1,\cdots, n$. Define $X^0$ to be the set of virus or infection free states: 
\[
X^0: = \{x\geq 0: x_i 0,~  \text{for all} ~ i = 1,\cdots, m \}.
\]
Let $\mathcal{F}_i (t,x)$ be the input rate of newly infected cells 
in the $i^{\text{th}}$ compartment. Let $\mathcal{V}_i^+(t,x)$ and 
$\mathcal{V}_i^-(t,x)$ respectively be the input rate of cells by other 
means and rate of transfer of cells out of a compartment $i$. Thus 
the system of Equations (\ref{eq:uninfected})--(\ref{eq:drug}) take the form 
\[
\frac{dx_i}{dt} = \mathcal{F}_i(t,x)-\mathcal{V}_i (t,x). 
\] 
We now proceed to prove the assumptions  \textbf{ $A_1 -A_5$}: 
	
\textbf{$A_1$}: For each $1\leq i \leq n$, the functions $\mathcal{F}_i(t,x)$, 
$\mathcal{V}_i^+$ and $\mathcal{V}_i^-$ are non-negative and continuous on 
$\mathbb{R} \times \mathbb{R}^n_+$ and continuously differentiable with 
respect to $x$. From Equations (\ref{eq:uninfected})--(\ref{eq:drug})
\begin{equation}
\label{eq:fv}
\mathcal{F}_i = \left( \begin{array}{c}
\frac{\beta U V}{K_u + U } \\ 
0 
\end{array} \right),
\quad \mathcal{V}_i^+ 
=   \left( 
\begin{array}{c}
\delta I + \nu_I E_T I + \frac{\delta_I I C}{K_c +C} \\ 
\gamma v 
\end{array} \right), \quad  \mathcal{V}_i^- =   \left( \begin{array}{c}
\tau E_V E_V I \\ 
b \delta  
\end{array} \right).
\end{equation}
Clearly, $\mathcal{F}_i$, $\mathcal{V_i}_i^+$ and $\mathcal{V}_i^-$ satisfy $A_1$. 
	
\textbf{$A_2$}: If $x_i = 0$, then $\mathcal{V}_i=0$. Clearly, from (\ref{eq:fv}), 
if the state variables are all equal to zero, then $\mathcal{V}_i=0$. 
	
\textbf{$A_3$}: $\mathcal{F}_i=0$ for $i>m$. This can be seen from the compartments $U$, $E_V$ and $E_T$.
	
\textbf{$A_4$}: If $x\in X^0$ , then $\mathcal{F}_i(x) = \mathcal{V}^+_i (x) = 0$ for $i = 1,\cdots , m$. 
At $X^0$, $V$ and $I$ are zero, implying that $\mathcal{F}_i = \mathcal{V}^+_i =0 $.   
	
\textbf{$A_5$}: If $\mathcal{F}_i=0$, then all the eigenvalues of $D \mathcal{V} (X^0)$ 
have positive real parts. This can be ascertained from Equation (\ref{eq:fv}).  
\end{proof}

\noindent C.2. Proof of Theorem~\ref{thm:R0-constantinfusion}

\begin{proof}
From Equation (\ref{eq:fv}),
\[
F = \left( 
\begin{array}{cc}
0 &  \frac{\beta U}{K_u + U } \\ 
0  & 0 
\end{array} \right),
\quad  V =  \left( 
\begin{array}{cc}
\delta +\nu_IE_T + \frac{\delta_I C^*}{K_c + C}+\tau E_v & 0 \\ 
b\delta  & \gamma 
\end{array} \right), 
\]
\[
~ FV^{-1} = \frac{1}{\gamma(\delta + \nu_I E_T^* 
+\frac{\delta_IC^*}{K_c + C^*})}\left( 
\begin{array}{cc}
\frac{b \delta \beta U}{K_U + U } 
& \frac{b\beta U^*}{(K_u + U^*)} \\ 
0 & 0\end{array} \right).
\]
Thus 
\begin{align}
\label{eq:Ro-proof}
R_0 = \rho(FV^{-1}) 
=  \frac{b\beta \delta U^*(K_c + C^*)}{\gamma \big[(K_c+ C^*) 
(\nu_I E_T^* +\delta) + \delta_I C^* \big]\big(K_u + U^* \big)}.
\end{align}
\end{proof}


\section{Appendix D}\label{appendixD} ~~~

\medskip

\noindent D.1. Elasticity indices of $R_0$

{\footnotesize 
\begin{align}
\begin{split}
e_b & = 1\\
e_{\beta} & = 1\\
e_{\gamma} & = -1\\
e_{\delta} & = \newcommand{\Bold}[1]{\mathbf{#1}}-\frac{{\Gamma_1\left(K K_{c} \alpha \psi 
+ K_{c} K_{u} \alpha \psi + K \alpha q + K_{u} \alpha q - K \delta_{U} q\right)}  
{\left(K_{c} \delta \psi + \delta q + \delta_{I} q\right)} \gamma}{{{\left(K_{c} 
\alpha \psi + \alpha q - \delta_{U} q\right)} {\left(K_{c} \psi + q\right)} K b \beta}}\\ 
e_{\alpha} & = \newcommand{\Bold}[1]{\mathbf{#1}}-\frac{{\Gamma_2\left(K K_{c} \alpha \psi 
+ K_{c} K_{u} \alpha \psi + K \alpha q + K_{u} \alpha q - K \delta_{U} q\right)} {\left(K_{c} \delta \psi 
+ \delta q + \delta_{I} q\right)} \alpha \gamma}{{\left(K_{c} \alpha \psi + \alpha q 
- \delta_{U} q\right)} {\left(K_{c} \psi + q\right)} K b \beta \delta}\\
e_{\delta_U} & = \newcommand{\Bold}[1]{\mathbf{#1}}\frac{{\Gamma_3\left(K K_{c} 
\alpha \psi + K_{c} K_{u} \alpha \psi + K \alpha q + K_{u} \alpha q 
- K \delta_{U} q\right)} {\left(K_{c} \delta \psi + \delta q 
+ \delta_{I} q\right)} \delta_{U} \gamma}{{\left(K_{c} \alpha \psi + \alpha q 
- \delta_{U} q\right)} {\left(K_{c} \psi + q\right)} K b \beta \delta}\\
e_{\delta_I} & = \newcommand{\Bold}[1]{\mathbf{#1}}-\frac{\delta_{I} q}{K_{c} 
\delta \psi + \delta q + \delta_{I} q}\\
e_{K_c} & = \newcommand{\Bold}[1]{\mathbf{#1}}\frac{\Gamma_4 K_{c} \psi q}{{\left(K K_{c} 
\alpha \psi + K_{c} K_{u} \alpha \psi + K \alpha q + K_{u} \alpha q 
- K \delta_{U} q\right)} {\left(K_{c} \alpha \psi + \alpha q - \delta_{U} 
q\right)} {\left(K_{c} \delta \psi + \delta q 
+ \delta_{I} q\right)} {\left(K_{c} \psi + q\right)}} \\
e_{K_u} & = \newcommand{\Bold}[1]{\mathbf{#1}}-\frac{{\left(K_{c} \psi 
+ q\right)} K_{u} \alpha}{K K_{c} \alpha \psi + K_{c} K_{u} \alpha \psi 
+ K \alpha q + K_{u} \alpha q - K \delta_{U} q}\\
e_{\psi} & = \newcommand{\Bold}[1]{\mathbf{#1}}\frac{\Gamma_4 K_{c} 
\psi q}{{\left(K K_{c} \alpha \psi + K_{c} K_{u} \alpha \psi + K \alpha q 
+ K_{u} \alpha q - K \delta_{U} q\right)} {\left(K_{c} \alpha \psi + \alpha q 
- \delta_{U} q\right)} {\left(K_{c} \delta \psi + \delta q 
+ \delta_{I} q\right)} {\left(K_{c} \psi + q\right)}} \\
e_q & = \newcommand{\Bold}[1]{\mathbf{#1}}-\frac{\Gamma_4 K_{c} 
\psi q}{{\left(K K_{c} \alpha \psi + K_{c} K_{u} \alpha \psi + K \alpha q 
+ K_{u} \alpha q - K \delta_{U} q\right)} {\left(K_{c} \alpha \psi + \alpha q 
- \delta_{U} q\right)} {\left(K_{c} \delta \psi + \delta q 
+ \delta_{I} q\right)} {\left(K_{c} \psi + q\right)}},
\end{split}
\end{align}}
where
\begin{align*} 
\Gamma_1 &= \frac{{\left(K_{c} \alpha \psi + \alpha q - \delta_{U} q\right)} {\left(K_{c} \psi 
+ q\right)}^{2} K b \beta \delta}{{\left(K K_{c} \alpha \psi + K_{c} K_{u} \alpha \psi 
+ K \alpha q + K_{u} \alpha q - K \delta_{U} q\right)} {\left(K_{c} \delta \psi 
+ \delta q + \delta_{I} q\right)}^{2} \gamma}\\
& \quad \quad \quad - \frac{{\left(K_{c} \alpha \psi + \alpha q 
- \delta_{U} q\right)} {\left(K_{c} \psi + q\right)} K b \beta}{{\left(K K_{c} \alpha \psi 
+ K_{c} K_{u} \alpha \psi + K \alpha q + K_{u} \alpha q - K \delta_{U} q\right)} 
{\left(K_{c} \delta \psi + \delta q + \delta_{I} q\right)} \gamma}\\
\Gamma_2 & = \frac{{\left(K K_{c} \psi + K_{c} K_{u} \psi + K q 
+ K_{u} q\right)} {\left(K_{c} \alpha \psi + \alpha q - \delta_{U} q\right)} 
{\left(K_{c} \psi + q\right)} K b \beta \delta}{{\left(K K_{c} \alpha \psi 
+ K_{c} K_{u} \alpha \psi + K \alpha q + K_{u} \alpha q - K \delta_{U} q\right)}^{2} 
{\left(K_{c} \delta \psi + \delta q + \delta_{I} q\right)} \gamma}\\
& \quad \quad \quad - \frac{{\left(K_{c} \psi + q\right)}^{2} 
K b \beta \delta}{{\left(K K_{c} \alpha \psi + K_{c} K_{u} \alpha \psi + K \alpha q 
+ K_{u} \alpha q - K \delta_{U} q\right)} {\left(K_{c} \delta \psi 
+ \delta q + \delta_{I} q\right)} \gamma} \\
\Gamma_3 & = \frac{{\left(K_{c} \alpha \psi + \alpha q - \delta_{U} q\right)} 
{\left(K_{c} \psi + q\right)} K^{2} b \beta \delta q}{{\left(K K_{c} \alpha \psi 
+ K_{c} K_{u} \alpha \psi + K \alpha q + K_{u} \alpha q - K \delta_{U} q\right)}^{2} 
{\left(K_{c} \delta \psi + \delta q + \delta_{I} q\right)} \gamma}\\
& \quad \quad \quad - \frac{{\left(K_{c} \psi + q\right)} K b \beta 
\delta q}{{\left(K K_{c} \alpha \psi + K_{c} K_{u} \alpha \psi + K \alpha q 
+ K_{u} \alpha q - K \delta_{U} q\right)} {\left(K_{c} \delta \psi 
+ \delta q + \delta_{I} q\right)} \gamma}\\
\Gamma_4 & = K K_{c}^{2} \alpha^{2} \delta_{I} \psi^{2} + K_{c}^{2} 
K_{u} \alpha^{2} \delta_{I} \psi^{2} + K_{c}^{2} K_{u} \alpha \delta 
\delta_{U} \psi^{2} + 2 \, K K_{c} \alpha^{2} \delta_{I} \psi q \\
&\quad \quad \quad + 2 \, K_{c} K_{u} \alpha^{2} \delta_{I} \psi q 
+ 2 \, K_{c} K_{u} \alpha \delta \delta_{U} \psi q  - 2 \, K K_{c} 
\alpha \delta_{I} \delta_{U} \psi q + K \alpha^{2} \delta_{I} q^{2} \\
& \quad \quad \quad + K_{u} \alpha^{2} \delta_{I} q^{2} + K_{u} \alpha 
\delta \delta_{U} q^{2} - 2 \, K \alpha \delta_{I} \delta_{U} q^{2} 
+ K \delta_{I} \delta_{U}^{2} q^{2}.
\end{align*}

\noindent D.2. Elasticity indices of endemic equilibrium.

Consider a system of differential equations dependent on a parameter $p$
\begin{align}
x^{\prime} = f(x,p) 
\end{align}
where 
\[
f(x,p) = \left(   
\begin{array}{c}
f_1(x,p) \\ 
f_2(x,p) \\ 
.\\
.\\
.\\
f_n(x,p)
\end{array}\right), 
\quad  x = \left(   
\begin{array}{c}
x_1 \\ 
x_2 \\ 
.\\
.\\
.\\
x_n
\end{array}\right). 
\]
At equilibrium point $x^* = x^*(p)$ satisfies $f(x^*(p),p) = 0$. 
Differentiating with respect to $p$, one obtains
\begin{align*}
\frac{\partial f}{\partial x_1} \left(x^*(p),p \right)x^{\prime}_1(p) 
+ \frac{\partial f}{\partial x_2} \left(x^*(p),p \right)x^{\prime}_2(p) 
+\cdots + \frac{\partial f}{\partial x_n} \left(x^*(p),p \right)x^{\prime}_n(p) 
+ \frac{\partial }{\partial p} \left(x^*(p),p \right) = 0. 
\end{align*}
This implies that 
\begin{align*}
\left(   \begin{array}{cccccc}
\frac{\partial f_1}{\partial x_1} & \frac{\partial f_1}{\partial x_1} 
& . &. &. & \frac{\partial f_1}{\partial x_n}  \\ 
\frac{\partial f_2}{\partial x_2}  & \frac{\partial f_2}{\partial x_2} 
& . & . & . & \frac{\partial f_2}{\partial x_n} \\ 
. & . & . & . &. & .\\
. & . & . & . &. & .\\
. & . & . & . &. & . \\   
\frac{\partial f_1}{\partial x_n} & \frac{\partial f_n}{\partial x_2} 
& . &. &. & \frac{\partial f_n}{\partial x_n}
\end{array}\right) \left(   
\begin{array}{c}
x^{\prime}_1 (p) \\ 
x^{\prime}_2 (p)\\ 
.\\
.\\
.\\
x^{\prime}_n (p)
\end{array}\right) = -\left(   
\begin{array}{c}
\frac{\partial f_1}{\partial p} \\ 
\frac{\partial f_2}{\partial p}\\ 
.\\
.\\
.\\
\frac{\partial f_n}{\partial p}
\end{array}\right)
\end{align*}

\begin{align*}
\left(   \begin{array}{c}
x^{\prime}_1 (p) \\ 
x^{\prime}_2 (p)\\ 
.\\
.\\
.\\
x^{\prime}_n (p)
\end{array}\right) = -J^{-1}_{x^*_p} \left(   
\begin{array}{c}
\frac{\partial f_1}{\partial p} \\ 
\frac{\partial f_2}{\partial p}\\ 
.\\
.\\
.\\
\frac{\partial f_n}{\partial p}
\end{array}\right).
\end{align*}
Multiplying both sides with the diagonal matrix, $diag(p/x^*_i)$, matrix 
\begin{align*}
\begin{pmatrix}
\frac{p}{x^*_1} & & \\
& \ddots & \\
& & \frac{p}{x^*_n}
\end{pmatrix}  \left(   \begin{array}{c}
x^{\prime}_1 (p) \\ 
x^{\prime}_2 (p)\\ 
.\\
.\\
.\\
x^{\prime}_n (p)
\end{array}\right) = -\begin{pmatrix} 
\frac{p}{x^*_1} & & \\
& \ddots & \\
& & \frac{p}{x^*_n}
\end{pmatrix} J^{-1} \left(   \begin{array}{c}
\frac{\partial f_1}{\partial p} \\ 
\frac{\partial f_2}{\partial p}\\ 
.\\
.\\
.\\
\frac{\partial f_n}{\partial p}
\end{array}\right).
\end{align*}
The sensitivity indices of the steady state variables at equilibrium 
$x^*_i$ with respect to a parameter $p$ are therefore given by
\begin{align}
\label{eq:derived-sensitivity}
\left(   \begin{array}{c}
\Gamma^{x^*_1}_p \\ 
\Gamma^{x^*_2}_p \\ 
.\\
.\\
.\\
\Gamma^{x^*_n}_p 
\end{array}\right)  
= -K J^{-1} \frac{\partial f}{\partial p} \left(x^*(p),p \right).
\end{align}
\end{appendices}


\section*{Acknowledgments} 

Joseph Malinzi was jointly supported 
by the University of Pretoria and DST/NRF SARChI Chair 
in Mathematical Models and Methods in Bioengineering and Biosciences.

Amina Eladdadi and K.A. Jane White would like to acknowledge and thank the 
UK-QSP Network (Grant-EP/N005481/1) for their financial support to attend 
the QSP-1st Problem Workshop and collaborate on this research. 

Torres was supported by FCT through CIDMA, project UID/MAT/04106/2013, 
and TOCCATA, project PTDC/EEI-AUT/2933/2014, 
funded by FEDER and COMPETE 2020.



\medskip

Submitted 27-March-2018; revised 04-July-2018; accepted 10-July-2018.

\medskip


\end{document}